\documentclass[english]{article}
\usepackage[T1]{fontenc}
\usepackage[latin9]{inputenc}
\usepackage{amsmath}
\usepackage{amssymb}
\usepackage{babel}
\usepackage{amsthm}
\usepackage{graphicx}
\usepackage{tikz}
\usepackage[all]{xy}
\usepackage{bm}
\usepackage{comment}
\usetikzlibrary{calc,patterns,angles,quotes}
\usepackage[toc,page]{appendix}
\usepackage{hyperref}

\newtheorem{prop}{Proposition}
\newtheorem{theorem}{Theorem}

\newtheorem{cor}{Corollary}

\newcommand{\R}{\mathbb{R}}

\newcommand{\C}{\mathbb{C}}

\begin{document}

	\title{Projective Integrable Mechanical Billiards}
	\author{Airi Takeuchi, Lei Zhao}
	\date{}
	\maketitle
	
\begin{abstract}
	In this paper, we use the projective dynamical approach to
	integrable mechanical billiards as in \cite{zhao2021} to establish the
	integrability of natural mechanical billiards with the Lagrange problem,
	which is the superposition of two Kepler problems and a Hooke problem,
	with the Hooke center at the middle of the Kepler centers, as the
	underlying mechanical systems, and with any combinations of confocal
	conic sections with foci at the Kepler centers as the reflection wall,
	in the plane, on the sphere, and in the hyperbolic plane. This covers
	many previously known integrable mechanical billiards, especially the
	integrable Hooke, Kepler and two-center billiards in the plane, as has
	been investigated in \cite{Takeuchi-Zhao}, as subcases. The approach of
	\cite{Takeuchi-Zhao} based on conformal correspondence has been also applied
	to integrable Kepler billiards in the hyperbolic plane to illustrate
	their equivalence with the corresponding integrable Hooke billiards on
	the hemisphere and in the hyperbolic plane as well.
\end{abstract}

%\tableofcontents

%atode kakitasu
\section{Introduction}
A two-dimensional mechanical billiard system $(M, g, U, \mathcal{B})$ is defined on a two-dimensional Riemannian manifold $(M, g)$ with a piecewise smooth curve $\mathcal{B} \subset M$ playing the role of a reflection wall and with $U: M \to \R$ an openly and densely defined smooth force function on $M$ determining a natural mechanical system whose equation of motion is
\begin{equation}\label{eq: Newton} \nabla_{\dot{q}}\dot{q}=\nabla U(q).
\end{equation}
 A particle moves according to the underlying force field and gets reflected elastically at $\mathcal{B}$, \emph{i.e.} at the point of reflection the tangential component of the velocity does not change while the normal component change its signs. The kinetic energy $K(q, \dot{q})=\dfrac{1}{2} g_{q} (\dot{q}, \dot{q})$ is invariant under elastic reflections, and thus the total energy $E(q, \dot{q})=K(q, \dot{q})-U(q)$ as well. 
 
A two-dimensional mechanical billiard is called integrable, if there exists an additional first integral, \emph{i.e.} a first integral of the underlying mechanical system invariant under the reflections, independent of the total energy $E$. Note that to address the problem of integrability, we do not insist on that the billiard mappings is always well-defined.

{Some examples of integrable mechanical billiards are known:}

{For the free motion in the plane $\R^2$, the billiards with a circular or an elliptic reflection wall have well-defined billiard mappings and are integrable. The integrability of the circular case is very easy to check since the angle of reflection is preserved. The integrability of the elliptic case has been shown by Birkhoff \cite{Birkhoff}. This integrability result can be extended also to the billiard system on the two-dimensional sphere and the two-dimensional hyperbolic plane \cite{Veselov-Alexander}\cite{Tabachnikov3}. }
	
{There are also known integrable billiard examples in the presence of non-constant force functions, the most studied systems are those defined with the Hooke or the Kepler problems.  }

{The Hooke problem and the Kepler problem in the plane $\R^2$ refer to the case when $U = f r^2$ and $U = m/r$ respectively, where $r$ is the distance of the particle from the fixed center $O \in \R^2$ and $f, m\in \R$ are parameters which we assume can take both signs: The force may be either attractive or repulsive.}

{For the Hooke problem in $\R^2$, it is rather direct to check that the systems with any line as a reflection wall are integrable. Also the one with a centered conic section as a reflection wall is integrable, in which the case of centered ellipse follows directly from the classical work of Jacobi \cite{Jacobi Vorlesung}. Recent work by Pustovoitov \cite{Pustovoitov2019}\cite{Pustovoitov2021} showed that any confocal combination of centered conic sections are also integrable. }

{For the Kepler problem in $\R^2$, the billiard systems with a line reflection wall which is not passing the center were proposed by Boltzmann \cite{Boltzmann}. The integrability of such systems has been established recently by Gallavotti and Jauslin \cite{Gallavotti-Jauslin} with an analysis on the geometry of ellipses, with alternative proofs by \cite{Felder} and \cite{zhao2021}.  %Moreover, such integrable systems can be generalized to $\mathbb{S}^2$ and $\mathbb{H}^2$ [REF].
In \cite{Takeuchi-Zhao}, we establish that any conic sections focused at the center and any confocal combination of them are also integrable, by using the classical Hooke-Kepler correspondence. }

{As compared to the Hooke and the Kepler problems, Euler's two-center problem in $\R^{2}$ are not super-integrable and the billiard problem defined by it seems to be less studied. In \cite{Takeuchi-Zhao}, we showed the integrability of such billiards with combinations of confocal conic sections reflection walls. }

In this paper, we explain that certain integrable mechanical billiards in the two-dimensional plane and constant curvature surfaces are related by projective correspondences. This allows us to yet extend some of our previous results in {\cite{Takeuchi-Zhao}} concerning integrable mechanical billiards in the plane with further extensions to surfaces of constant curvatures. 
%We note that for free billiards the projective approach has been addressed in the works of Tabachnikov 

%Firstly, we explain projective correspondences between planer and spherical/pseudo-spherical integrable billiard systems. 
{Our main methodology in this paper is based on the projective correspondence between mechanical billiards. This means that in addition to the projective correspondence of the underlying natural mechanical systems, also the laws of reflection are in correspondence to each other, so that a billiard trajectory in one system is projected to a billiard trajectory in the other system. The energies of the systems then give rise to a pair of independent first integrals for both of the two billiard systems.  With this method, the projective correspondence between integrable planar and spherical Kepler billiards with a line or centered circle reflection wall was presented in \cite{zhao2021}. The method can be thought of as an adaptation of the projective method for geodesic flows and free billiards as in \cite{Topalov-Matveev}, \cite{Tabachnikov2}, \cite{Tabachnikov3}, \cite{Tabachnikov 1999}, \cite{Tabachnikov 2002} to the case of mechanical billiards.}

In this paper we consider the billiard systems defined through the Lagrange problem in the plane with $U = m_1/{r_1} +m_2/{r_2} + f r^2 $, on a sphere with $U= m_1 \cot \theta_{Z_1} + m_2 \cot \theta_{Z_2} + f \tan^2 \theta_{Z_{mid}}$ and in a hyperbolic plane with $U= m_1 \coth \theta_{Z_1} + m_2 \coth \theta_{Z_2} + f \tanh^2 \theta_{Z_{mid}}$, which are the problems of adding an elastic force to the two-center problem defined on such a space %\marginpar{\rd{Can you write out the force function explicitly?}
centered at the middle of the two centers. The precise definitions of the notations are given in Section \ref{sec: projective properties of Hooke Kepler Lagrange problems} and Section \ref{sec: hyperbolic Lagrange billiards}. This integrable system has been identified by Lagrange \cite{Lagrange} in the planer case. Note that such a system is singular at the Kepler centers, as well as a singular set created by the elastic force on the sphere, and is regular elsewhere. 

By setting some of the mass factors to zero we get several systems as particular cases including the two-center problem, the Kepler problem, and the Hooke problem in the plane, on a sphere, and in a hyperbolic plane. By confocal conic sections we shall mean those with the two Kepler centers as foci.  

\begin{theorem}\label{thm: Theorem 1} The mechanical billiard problems defined in the plane, on a sphere and in a hyperbolic plane with the Lagrange problem and with any combination of confocal conic sections with foci at the two Kepler centers as reflection wall, are integrable.
\end{theorem}

{In the plane, the billiard problems defined through the Hooke, the Kepler, and the two-center problems with combinations of confocal conic sections are therefore subcases of Theorem \ref{thm: Theorem 1} and thus their integrability directly follows. These have been previously discussed via a different method, based on conformal transformations, in \cite{Takeuchi-Zhao}. Theorem \ref{thm: Theorem 1} provides an alternative proof of their integrability as well as extensions to the sphere and the hyperbolic plane.} 

Note that somehow in contrast to the conformal transformation used in \cite{Takeuchi-Zhao}, this projective method can be directly applied to the case of higher dimensional problems, and will always provide two first integrals for the Lagrange problems. We shall not discuss these higher dimensional problems in this article and will leave it for future works. Restricting to dimension 2 raises the question of whether some of the integrable systems can be indeed also related by conformal transformations. Toward the end of this article we shall present such links of integrable Kepler billiards in the hyperbolic plane, and the integrable Hooke billiards defined on the sphere and in the hyperbolic plane. {In Proposition \ref{prop: stereographic_confocal}, we also show that a family of confocal focused hyperbolic conic sections are transformed into a family of confocal centered spherical/hyperbolic conic sections by the complex square mapping in conformal charts, which might have an independent geometrical interest.}

%\rd{REFERENCES ON PROJECTIVE BILLIARDS.}
%\bl{. The projective dynamics In the presence of the Kepler potential and more general class of gravitational potential are explained in [REF]. The The projective properties of the two-centers (Lagrange?) problem has been presented in [REF]. Our result can be considered as an extension of this method to the Lagrange billiard systems.}

%\rd{DESCRIPTION OF SECTIONS.}
{We organize this paper as follows:}

{{In Section \ref{sec: principles of projective dynamics}, we explain the settings and the principle properties of projective dynamics and define projective correspondences of billiard systems.}  {In Section \ref{sec: projective properties of Hooke Kepler Lagrange problems}, we recall projective properties of the Hooke and the Kepler problems and their spherical/hyperbolic analogous systems. We get the projective property of Lagrange problems as has been discovered in \cite{Albouy1}.}{ In Section \ref{sec: integrable Lagrange billiards}, we prove Theorem \ref{thm: Theorem 1} for the planar and the spherical cases, and we discuss some subcases.}
{In Section \ref{sec: hyperbolic Lagrange billiards}, we briefly discuss the hyperbolic case and establish Theorem \ref{thm: Theorem 1} for this case.}
{The conformal correspondences among the Hooke and the Kepler billiards in the hyperbolic space and the Hooke billiard on the hemisphere are discussed in Section \ref{sec: Hooke-Kepler conformal correspondence }. }}%\marginpar{\rd{Please label the Sections and refer to them in this paragraph.}}

\section{Principles of Projective Dynamics}
\label{sec: principles of projective dynamics}

Let $(M, g, U)$ be a natural mechanical system. The system \emph{possesses a corresponding system} if there exists another natural mechanical system $(M, g', U')$ such that they have the same orbits in $M$ up to time-parametrizations. In this case, any first integral of $(M, g, U)$ is also a first integral of $(M, g', U')$ and vice versa. In particular, the energy $E'$ of the second system $(M, g', U')$ is a first integral of $(M, g, U)$. When $E'$ is functional independent from the energy $E$ of $(M, g, U)$ we have an additional first integral of the system $(M, g, U)$. The same can be said for the system  $(M, g', U')$ in a completely similar way. In practice the underlying smooth manifolds may not be identical. In this case we assume them to be diffeomorphic and identify them by a proper diffeomorphism. 

The subject of projective dynamics is to study correspondences of natural mechanical systems induced by projections. To explain further we write the equations of motion of a particle on a Riemannian manifold $(M, g)$ moving in a force field $F$, as
\begin{equation} \nabla_{\dot{q}} \dot{q}=F(q).
\end{equation}
Note that when $F(q)=\nabla U(q)$ is the gradient of a force function, then we say that the system is \emph{derived from a potential}. By definition the potential is the negative of the force function. The procedure of the projection from $M$ to $M'$ by a diffeomorphism $\phi: M\mapsto M'$ with a time reparametrization factor $\rho: M' \mapsto \R$ is that a force field $F$ on $M$ is projected into the force field $F':=\rho \cdot \phi_{*} F$.

In words:

\emph{The projection of the force field of a system is the force field given by the push-forward of the projection multiplied with a time reparametrization factor. }

% Some prominent examples is given by the central projection from the unit sphere  $\mathcal{S} \subset \R^{3}$ to a plane, say $V:=\R^{2} \times \{-1\} \in \R^{3}$, which is known to give any planar Kepler problem in $V$, up to a proper affine transformation $V \mapsto V$, and the planar Hooke problem in $V$ centered at $(0, 0, -1)$ a corresponding system, namely from their analogous systems on the open hemisphere $\mathcal{S}_{SH}:=S \cap \R^{2} \times \R_{-}$. We call these systems the hemispherical Kepler and Hooke problems respectively. 
 
%Precisely speaking, a common time change uniquely determined by the projection is applied to the push-foward of the force fields of the spherical problem and results in the force fields of the planar problem. Note that this determines the right hand side, \emph{i.e.} the force field  in Newton's equation \eqref{eq: Newton}, which might not be derived from a potential for a given metric. An important observation is that the left-hand side of  \eqref{eq: Newton} does not in general determines the metric uniquely, and therefore any metric giving the same form of the left-hand side of  \eqref{eq: Newton}  can be used and the system is derived from a potential if one of such choice of the metric makes the right-hand side of  \eqref{eq: Newton} the gradient of some function: This freedom of choosing different metric is essential for the applicability of this theory.

Now arguing with force fields defined on a manifold, we have the following {principle of superposition}: 

\emph{The projection of superposition of the force fields is the superposition of the projections of the force fields.}

When the force fields are derived from potentials, then so is their superposition. In general, the projections of these force fields are no longer derived from potentials. However this indeed holds for special systems that we are going to address in this article, which provide corresponding systems to the original systems. 

%Consequently, when all the involved force fields are derived from potentials, then this principle implies, in the case of the sphere-plane central projection, that:

%\emph{The system with a summation of potentials on the hemisphere projects to the system with the summation of potentials of individual projected systems, provided that the projected metrics are in common. }

We now comment on billiard correspondences. For this it seems convenient to identify the base manifold $M$ by a diffeomorphism and consider the law of reflection in $M$ with respect to the two metrics. The tangential direction is free from the choice of the metric but the normal direction depends on the metric, and therefore a priori the elastic laws of reflections with respect to different metrics are different. We say that there is a billiard correspondence when the elastic laws of reflection agree in addition to the correspondence of underlying natural mechanical systems. As we can see, this depends on the choice of metric and the shape of the reflection wall. When there is a billiard correspondence, then the billiard trajectories, ignoring time parametrizations, correspond to each other by projection and therefore their billiard mappings are equivalent.  Conserved quantities of one system are thus transformed into conserved quantities of the other system, and therefore the integrability of the billiard system also carries over.  

%Very much the same thing can be said for the plane-hyperboloid correspondence, for which we realize a sheet of the hyperboloid as a model for the hyperbolic plane, with both embedded in the $(2,1)$-Minkowski space $\R^{2, 1}$ in an analogous way.

We refer to \cite{Albouy2}, \cite{Albouy1}, \cite{zhao2021} for further, more detailed presentatons of projective dynamics.

\section{Projective Properties of the Hooke, Kepler and Lagrange Problems}
\label{sec: projective properties of Hooke Kepler Lagrange problems}

In this section, we discuss some projective properties of the Hooke and the Kepler problems and their spherical/hyperbolic analogous systems that we need. Then we shall show that the Lagrange problem also has spherical/hyperbolic analogous systems by the principle of superposition of projective dynamics, which then gives to each of these systems a pair of independent first integrals including their own energies.

%\subsection{Planar and Spherical Hooke and Kepler Problems}

\subsection{The Hemisphere-Plane Projection}\label{Subsection: The Hemisphere-Plane Projection}
We set $V=\R^{2} \times \{-1\} \subset \R^{3}$ and $\mathcal{S} \subset \R^{3}$ the unit sphere in $\R^{3}$. The central projection from the origin of $\R^{3}$ projects the open south-hemisphere $\mathcal{S}_{SH}$ onto the plane $V$. We equip $\mathcal{S}$ and $\mathcal{S}_{SH}$ with their induced round metrics from $\R^{3}$, while on $V$ we allow an affine change of metric. A force field $F_{V}$ on $V$ is carried to a force field $\hat{F}_{S}$ on $\mathcal{S}_{SH}$ by the push-forward of the central projection, which is consequently reparametrized into another force field $F_{S}$ with the factor of time change uniquely determined by the projection. 

The Euclidean norm of $\R^{3}$ as well as its restriction to $V$ is denoted by $\|\cdot\|$.

Let $q \in \mathcal{S}_{SH}$ be projected to $\tilde{q} \in V$ by the central projection:
	$$
	q = \| \tilde{q} \|^{-1} \tilde{q}.
	$$
	We write $\dot \,:=\dfrac{d}{dt}$ the time derivative. We start by the force field $F_{V}$ in $V$ and deduce the corresponding force field $F_{S}$ on $\mathcal{S}_{SH}$ which is equivalent to the other way around but the computation simplifies. The equation of motion of the system in $V$ is
	$$\ddot{\tilde{q}}=F_{V}(\tilde{q}).$$
	We compute
	$$
	\dot{q}  = \| \tilde{q} \|^{-2} (\dot{\tilde{q}}  \| \tilde{q} \| - \langle \nabla \| \tilde{q}\| , \dot{\tilde{q}} \rangle \tilde{q}).
	$$
	We now take a new time variable $\tau$ for the system on $\mathcal{S}_{SH}$, and $' :=\dfrac{d}{d\tau}$ such that 
	\begin{equation}
	\label{eq: time_parameter_change}
	{\dfrac{d}{d \tau}= \| \tilde{q} \|^{2}\dfrac{d}{d t}}. 
	\end{equation}
	We thus have
	$$
	q'  =  (\dot{\tilde{q}}  \| \tilde{q} \| - \langle \nabla \| \tilde{q}\| , \dot{\tilde{q}} \rangle \tilde{q}).
	$$
	and
	$$q''= \| \tilde{q} \|^{2} (\ddot{\tilde{q}}  \| \tilde{q} \| - (\langle \nabla \| \tilde{q}\| , \ddot{\tilde{q}} \rangle+\langle \overset{\cdot}{(\nabla \| \tilde{q}\|)} , \dot{\tilde{q}} \rangle) \tilde{q}).$$
	Consequently we have
	\begin{equation}\label{eq: second derivative correspondence}
	q''=\| \tilde{q} \|^{2} (F_{V}(\tilde{q})  \| \tilde{q} \|- \lambda(\tilde{q}, \dot{\tilde{q}}, \ddot{\tilde{q}}) \tilde{q})).
	\end{equation}
	in which we have set $\lambda(\tilde{q}, \dot{\tilde{q}}, \ddot{\tilde{q}})=\langle \nabla \| \tilde{q}\| , \ddot{\tilde{q}} \rangle+\langle \overset{\cdot}{(\nabla \| \tilde{q}\|)} , \dot{\tilde{q}} \rangle$.
	
	We observe that the first term of the right hand side of this equation depends only on $\tilde{q}$ and consequently depends only on $q \in \mathcal{S}_{SH}$ by central projection, {while the second term is radial.} Projecting both sides of this equation to the tangent space $T_{q}  \mathcal{S}_{SH}$ we get the equation of motion on $S_{SH}$, assuming the form
	$$\nabla_{q'} q'=F_{S}(q).$$
	
	%We may change the time parameter according to $\dfrac{d }{d  \tau} = \| q_1 \|^2 \dfrac{d}{dt}$, and take the corresponding time derivative
	%\[
	%q'_0 = \dot{q}_1  \|q_1\| - \langle \nabla \| q_1\| , \dot{q}_1 \rangle q_1.
	%\]

For our purpose, we would like to have natural mechanical systems which are centrally projected to natural mechanical systems, \emph{i.e.} the question is, when we start from a natural mechanical system on $S_{SH}$ resp. $V$, then whether the projected system on $V$ resp. $S_{SH}$ is also derived from a potential and is thus also a natural mechanical system. As we would expect this does not hold in general. Nevertheless, it actually holds for some important systems.

\subsection{Projective Properties of the Hooke and Kepler Problems}
We consider a central force problem $F_{S}$ on $\mathcal{S}$ with a distinguished center $Z \in \mathcal{S}_{SH}$. By assumption the force field is invariant under the $SO(2)$-action by rotations around $Z$ on $\mathcal{S}$. The projected force field $F_{V}$ on $V$ is in general not derived from a potential. In the same way, a central force problem $F_{V}$ in $V$ with a center $\tilde{Z}$ might not project to a system derived from a potential on $\mathcal{S}_{SH}$.

There are special cases that this does hold. The first is relatively easy to see: when $Z=(0,0, -1)$, the projected force field $F_{V}$ is also invariant under the  $SO(2)$-action by rotations in $V$ as inherited from rotations around the vertical axis in $\R^{3}$, and therefore $F_{V}$ is derived from a potential. The second case is maybe not as easy to see: The point $Z \in \mathcal{S}_{SH}$ can be chosen arbitrary, and $F_{V}$ will be derived from a potential when $F_{S}$ is the force field of the Kepler-Serret Problem on the sphere {\cite{Serret}}, and in this case $F_{V}$ itself is the force field of a Kepler problem in $V$ for a proper choice of an affine metric. Also, among the problems belonging to the first case, the Hooke problems have the property that $F_{V}$ is derived from a potential for \emph{any} affine metrics in $V$. 

\subsubsection{The Kepler Problems}

We first discuss the case of the Kepler problems. The Kepler-Serret problem, or the spherical Kepler problem, is the natural mechanical system $(\mathcal{S}, g_{st}, \hat{m} \cot \theta_{Z})$, in which $g_{st}$ is the round metric on the sphere, $\hat{m} \in \R$ is the mass-factor and $\theta_{Z}$ is the central angle the moving particle made with $Z$. The system naturally restricts to a natural mechanical system $(\mathcal{S}_{SH}, g_{st}, \hat{m} \cot \theta_{Z})$ by restriction. In the case that $Z$ is \emph{vertical}, $Z=(0,0, -1)$, it is not hard to see by a direct computation that the spherical Kepler problem is projected to the planar Kepler problem $(V, \|\cdot\|, \hat{m}/\|\tilde{q}\|)$. Consequently the orbits of the spherical Kepler problem are all conic sections on the sphere by means of orbital correspondence and analytic extension. A special property of the Kepler problem is that this remains true when $Z$ is not vertical, up to a change of metric and of the mass factor \cite{Halphen}. See also \cite{Albouy2}, \cite{zhao2021}.

To normalize the situation we set $Z=  \left(0, \dfrac{a}{\sqrt{1 + a^2}}, -\dfrac{1}{\sqrt{1 + a^2}}\right) \in \mathcal{S}_{SH}$ for $a \in \R$, and $\tilde{Z}=(0, a, -1)$ the projection point of $Z$ in $V$. For $\tilde{q}=(\tilde{x}, \tilde{y}, -1) \in V$ we define
      \begin{equation}
	\label{eq: metric_a}
	\| \tilde{q}\|_a= \sqrt{ \tilde{x}^2 + \frac{\tilde{y}^2}{1+a^2}}
	\end{equation}
	which is an affine change of norm from the induced norm on $V$ with origin at $(0,0,-1)$ of the standard Euclidean norm $\|\cdot\|$ in $\R^{3}$.
%{This situation is explained in the following proposition and this fact has been originally shown in \cite{zhao2021}} \marginpar{{I think it should be attribute to Higgs or Halphen}}

\begin{prop} 
	\label{prop: Kepler_Proj_Cor}
	{The spherical Kepler problem {$(\mathcal{S}_{SH}, g_{st}, \hat{m} \cot \theta_{Z})$} projects to the Kepler problem $(V, \|\cdot \|_{a}, m/\|\tilde{q}-\tilde{Z}\|_{a})$ such that {$m = \dfrac{\hat{m}}{\sqrt{1 + a^2}}$.}}
\end{prop}

\begin{proof}
	%{Without loss of generality we set the spherical Kepler center on $\mathcal{S}_{SH}$ at}
	%$$Z :=  \left(0, \dfrac{a}{\sqrt{1 + a^2}}, -\dfrac{1}{\sqrt{1 + a^2}}\right),$$
	%{which is projected to the planer Keplerian center in $V$ at }
	%\[
	%\tilde{Z}: = (0,a,-1).
	%\]
	%{
	%We define the norm on $V$ as
	%\begin{equation}
	%\label{eq: metric_a}
	%\| q\|_a= \sqrt{ x^2 + \frac{y^2}{1+a^2}}
	%\end{equation}
	%for all $q = (x,y) \in V$ when we regard $Z_0:= (0,0,-1)$ as the origin of $V$.	
	%}
	{With the procedure explained in Subsection \ref{Subsection: The Hemisphere-Plane Projection}, we arrive from a planar force field $F_{V}$ to a spherical force field $F_{S}$.}
	
	%Let $\tilde{q} \in V$ and $q \in S_{SH}$ be corresponded via the central projection:
	%\[
	%q = \| \tilde{q} \|^{-1} \tilde{q}.
	%\]
	%By taking {time derivative} of $q$, we obtain
	%\[
	%\dot{q}  = \dot{\tilde{q}}  \| \tilde{q} \|^{-1}  - \| \tilde{q} \|^{-2} \langle \nabla \| \tilde{q} \| , \dot{\tilde{q}} \rangle \tilde{q}.
	%\]
	%We apply time-change according to 
	%\begin{equation}
	%\label{eq:time_change_Kepler}
	%\dfrac{d }{d  \tau} = \| q \|^2 \frac{d}{dt},
	%\end{equation}
	%then the time derivative with respect to new time parameter $\tau$ is
	%\[
	%q' = \dot{\tilde{q}}  \|\tilde{q}\| - \langle \nabla \| \tilde{q}\| , \dot{\tilde{q}} \rangle \tilde{q}
	%\]
	%and %also the second time derivative becomes
	%\begin{equation}
	%\label{eq: sec_time_der}
	%q '' = \| \tilde{q}\|^3  \ddot{\tilde{q}} -  \| \tilde{v}\|^2 \langle \nabla \| \tilde{q}\| , \ddot{\tilde{q}} \rangle \tilde{q} -  \| \tilde{q}\|^2 \langle  \nabla^2 \| \tilde{q}\| ,  \dot{\tilde{q}} \dot{\tilde{Q}} \rangle  \tilde{q} ,
	%\end{equation}
	%with
	%\[
	%\dot{\tilde{Q}} = 
	%\begin{pmatrix}
	%\dot{x} & 0 & 0 \\
	%0 & \dot{y} & 0 \\
	%0 & 0 &  \dot{z} 
	%\end{pmatrix}
	%\]
	%where $\dot{\tilde{q}} = (\dot{x}, \dot{y},\dot{z})$.
	%The last two terms of the LHS of (\ref{eq: sec_time_der}) have the direction towards $\tilde{q}$ so as $q$, thus these components are vertical to $S_{SH}$.
	{We now consider the Kepler problem on $V$: 
	\begin{equation}
	\label{eq: Kepler_force_V}
	(V, \|\cdot \|_{a}, m \| \tilde{q} - \tilde{Z} \|_a^{-1})
	%m \| \tilde{q} - \tilde{Z} \|_a^{-1}.
	\end{equation}
	which determines the force field 
	\begin{equation}\label{eq: plane Kepler force field}
	F_{V} (\tilde{q}):=- m   \| \tilde{q} - \tilde{Z} \|_a^{-3} (\tilde{q}- \tilde{Z})
	\end{equation}
	 on $V$. }
	 
	 {We now plug \eqref{eq: plane Kepler force field} into the right hand side of \eqref{eq: second derivative correspondence} and compute its projection to the tangent space of $T_{q } \mathcal{S}_{SH}$. We may effectively forget the second term in the right hand side of \eqref{eq: second derivative correspondence} since it projects to zero in $T_{q } \mathcal{S}_{SH}$. As for the first term in the right hand side of \eqref{eq: second derivative correspondence}, we see that it is again central on $\mathcal{S}_{SH}$ by the central projection. Therefore it is enough to compute its norm to determine the corresponding $F_{S}$ on $\mathcal{S}_{SH}$. }
	 
	%{ For this purpose we first remark that $F_{S}$ is invariant under infinitesimal rotations on $\mathcal{S}_{SH}$ around the center $Z$. To see this, it is enough to check that the infinitesimal $\R$-action on $\mathcal{S}_{SH}$ by infinitesimal rotations around $Z$ are projected to infinitesimal $\R$-actions in $V$ under which the force field is invariant. }
	 
	%{ PLEASE CHECK THIS.}
	 
	%{ Therefore it is enough to determine $F_{S}$ when restricting to a great circle in $\mathcal{S}_{SH}$ passing through $Z$, which correspond to a line in $V$ by the central projection. We take the line to be $\{{0, y, -1}\}$, so the great circle is the intersection of $\mathcal{S}_{SH}$ with the plane $\{(0, y, z)\} \subset \R^{3}$. }
	
	%{ PLEASE DO THE COMPUTATION.}

	{For this purpose, we restrict the system to (oriented) planes passing through the centers $Z, \tilde{Z}$ as well as the center $O=(0,0,0)$ of $\mathcal{S}_{SH}$. These planes in $\R^{3}$ form an $S^{1}$-family. We compute the restricted force field on any of these planes.}
	
	{We fix such a plane {$W$, which necessarily intersects $V$ by construction. Let $\ell$ be the  intersection line. Let $G$ be the point on $\ell$ such that $OG$ is perpendicular to $\ell$. Let  $\phi$ be the angle between $\ell$ and the intersection line $\{(0, y, -1)\}$ of the $yz-$plane and $V$. The restriction to $\ell$ of the function $\|\tilde{q} - \tilde{Z}\|_a$ can be written as }
		\[
		\|\tilde{q} - \tilde{Z}\|_a = \| \tilde{q} - \tilde{Z}\| \sqrt{\sin^2 \phi + \dfrac{\cos^2 \phi}{1+ a^2}} =\| \tilde{q} - \tilde{Z}\| \sqrt{\frac{1 + a^2 \sin^2 \phi }{1 + a^2}}.  
		\]
		Thus the force filed $F_{V}(\tilde{q})$ restricted to $\ell$ {is given by}
		\[
		F_{V}(\tilde{q}) = -m { \left( \frac{1+a^2}{1 + a^2 \sin^2 \phi }  \right)^{3/2} \| \tilde{q} - \tilde{Z}\|^{-3}} (\tilde{q} - \tilde{Z}).
		\]
		{The line $\ell$ passes through the two points $\tilde{q} $ and $\tilde{Z} $ and the equation of $\ell$ is given by 
		\[
		y = \cot \phi \, x + a.
		\]
		Let $Z_0= (0,0,-1)$, then $G$ can be obtained as the point on $\ell $ such that $Z_0 G$ is perpendicular to $\ell$ and computed as
		\[
		G:= (-a \sin \phi \cos \phi, -a \sin^2 \phi,-1).
		\]
		}
		{By  \eqref{eq: second derivative correspondence}, the corresponding force field on  $\mathcal{S}_{SH}$ is determined by the projection of $\| \tilde{q}\|^3 F_{V}(\tilde{q}) $ to $T_{q} \mathcal{S}_{SH}$, which is computed as} 
		\[
		\| \tilde{q}\|^3 F_{V}(\tilde{q}) \cdot \cos \theta_G = \| \tilde{q}\|^3 F_{V}(\tilde{q}) \cdot \frac{\| G\|}{\|\tilde{q}\|}=\| \tilde{q}\|^3 F_{V}(\tilde{q}) \cdot \frac{\sqrt{1 + a^2 \sin^2 \phi}}{\|\tilde{q}\|}=\| \tilde{q}\|^2   \sqrt{1 + a^2 \sin^2 \phi} \cdot F_{V}(\tilde{q}) ,
		\] %\marginpar{{One should maybe present more detail on how to compute $\| G\|$. A figure would be helpful.}}
		where $\theta_G = \angle \tilde{q} O G$.
		
		{We now compute its norm as}
		\begin{align*}
		 &|m| \| \tilde{q}\|^2  \| \tilde{q} - \tilde{Z} \|^{-2} (1+a^2)^{3/2} (1 + a^2 \sin^2 \phi )^{-1} \\
		  =&  \sqrt{1+a^2} |m| \| \tilde{q}\|^2 \|\tilde{Z}\|^2  \| \tilde{q} - \tilde{Z} \|^{-2} (1 + a^2 \sin^2 \phi  )^{-1} \\
		  = & |\hat{m}|\| \tilde{q}\|^2 \|\tilde{Z}\|^2  \| \tilde{q} - \tilde{Z} \|^{-2} \| G\|^{-2} \\
		   = &|\hat{m}| \sin^{-2} \theta_{Z}.
		\end{align*}
		{if we set $m = \dfrac{\hat{m}}{\sqrt{1+a^2}}$.} For the last equality we applied the law of sines for the triangle $\tilde{q} O \tilde{Z}$. The computation is illustrated in Figure \ref{fig: section_views}.
	} 
	
		\begin{figure}
		\begin{tikzpicture}
		\begin{scope}[scale= 1.0]
		\draw[-,>=stealth] (-3.5,-2)--(3.5,-2)node[right]{$\ell$}; %y?
		\coordinate (O) at (0, 0);
		\fill[black] (O) circle (0.05);
		\coordinate (G) at (0, -2);
		\fill[black] (G) circle (0.05);
		\coordinate (Zt) at (-2, -2);
		\fill[black] (Zt) circle (0.05);
		\coordinate (qt) at (-3, -2);
		\fill[black] (qt) circle (0.05);
		\coordinate (W) at (-3,1) ;
		\fill[black] (W) circle (0.0) node{$W$};
		
		\draw (O) circle [radius=1.5];
		
		\coordinate (Z) at (-1.06, -1.06);
		\fill[black] (Z) circle (0.05) node[below]{$Z$};
		\coordinate (q) at (-1.24, -0.83);
		\fill[black] (q) circle (0.05) node[left]{$q$};
		
		\coordinate (R) at (2,-2);
		\coordinate (P) at ($(G)!.3cm!(R)!.3cm!90:(R)$);%??????????
		\draw[thick, red] ($(G)!(P)!(O)$)--(P);
		\draw[thick, red] ($(G)!(P)!(R)$)--(P);
		
		\draw (O)node[right]{O} -- (G) node[below]{G};
		\draw (O) -- (Zt) node[below]{$\tilde{Z}$};
		\draw (O) -- (qt) node[below]{$\tilde{q}$};
		\end{scope}
		\begin{scope}[xshift = 5.5cm, scale=0.6]

		\end{scope}
		
		\begin{scope}[xshift = 6.0cm, scale=1.0]
		\draw[-,>=stealth] (0,-3.5)--(0,2)node[right]{$\ell$}; 
		\coordinate (Z0) at (1, 0);
		\fill[black] (Z0) circle (0.05);
		\coordinate (G) at (0, 0);
		\fill[black] (G) circle (0.05);
		\coordinate (Zt) at (0, -2);
		\fill[black] (Zt) circle (0.05);
		\coordinate (qt) at (0, -3);
		\fill[black] (qt) circle (0.05);
		\coordinate (V) at (-1,1.5) ;
		\fill[black] (V) circle (0.0) node{$V$};
		
		\draw (Z0) node[right]{$Z_0$}  -- (G) node[left]{$G$};
		\draw (Z0) -- (Zt) node[left]{$\tilde{Z}$};
		\draw (Z0) -- (qt) node[left]{$\tilde{q}$};
		
		\draw pic[draw=black, "$\phi$", angle eccentricity=1.4, angle radius=0.5cm] {angle=Z0--Zt--G};
		
		\coordinate (R) at (0,2);
		\coordinate (P) at ($(G)!.3cm!(R)!.3cm!-90:(R)$);%??????????
		\draw[thick, red] ($(G)!(P)!(Z0)$)--(P);
		\draw[thick, red] ($(G)!(P)!(R)$)--(P);
		
		\end{scope}
		\end{tikzpicture}
		\caption{{Sectional Views on the Planes $W$ and $V$.}}
		\label{fig: section_views}
	\end{figure}
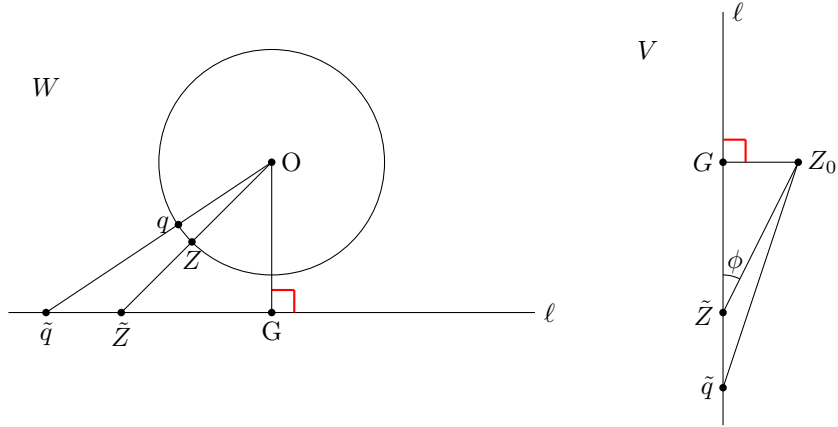

	{So after this computation we conclude that $F_{S}$ is the central force field on $\mathcal{S}_{SH}$ with strength $|\hat{m}| \sin^{-2} \theta_{Z}$ in which $\theta_{Z}$ is the central angle of $q$ to $Z$, pointing toward $Z$ or its antipodal point according to the sign of $\hat{m}$. This force field can be extended to the whole $\mathcal{S}$ which is singular only at $Z$ and its antipodal point, and is invariant under rotations along the line $OZ$. Restricting to a great circle passing through the point $Z$ we conclude that this system is derived from the force function $\hat{m} \cot \theta_{Z}$.}

\end{proof}

Note that among all homogeneous central force problems, this property of being projective invariant is unique for the Kelpler problem \cite{Albouy2}.

\subsubsection{The Hooke Problems}
The spherical Hooke problem is the system $(\mathcal{S}, g_{st}, f \tan^{2} \theta_{Z})$ with $f \in \R$. This is seen to be the analytic extension of the projection  $(\mathcal{S}_{SH}, g_{st}, f \tan^{2} \theta_{Z})$ of the Hooke problem in the plane $(V, \|\cdot\|, f \|\tilde{q}\|^{2})$. A special projective property of the Hooke problem is summarized in the following proposition. {In contrast to the Kepler case, here we assume} that the center for the Hooke problem is vertical \emph{i.e.} $Z = (0,0,-1)$.

\begin{prop} 
	\label{prop: Hooke_Proj_Cor}
	The spherical Hooke problem $(\mathcal{S}_{SH}, g_{st}, f \tan^{2} \theta_{Z})$ with $Z=(0,0,-1)$ projects to any of the Hooke problems in $V$ of the form $(V, \|\cdot|\|_{a}, f \|\tilde{q}\|_{a}^{2})$ for any $a \in \R$. 
\end{prop}

\begin{proof}
	 {The Hooke problem in $V$ with respect to a norm $\|\cdot\|_{a}$ is the system
	 \[
	 (V, \| \cdot \|_a, f \| \tilde{q}\|^2_a).
	 \]
	 The corresponding force field is given by 
	 \[
	 F_V (\tilde{q}):= 2 f (\tilde{q} - Z).
	 \]
	 A simple property which nevertheless worths to be mentioned, is that this force field is independent of $a$, \emph{i.e.} this force field corresponds to \emph{any} Hooke system of the above form.}
	 
	{The corresponding force field $F_S$ on $\mathcal{S}_{SH}$ is again determined by the central projection, and we obtain a central force field on  $\mathcal{S}_{SH}$ centered at $Z$, with the sign of $f$ determines whether $Z$ is attractive or repulsive just as in the planar case. }
	{Again, we just have to determine the norm of $F_{S}$. For this purpose, we restrict the system to any planes passing trough the center $Z$ and the center $O$ of $\mathcal{S}_{SH}$. With the same argument as in the proof of Proposition \ref{prop: Kepler_Proj_Cor}, we just have to compute  the projection of $\|\tilde{q}\|^3 F_V(\tilde{q})$ to the tangent space $T_q \mathcal{S}_{SH}$,  given by 
	 \[
	 2f \| \tilde{q}\|^3 \cos \theta_Z  (\tilde{q} - Z).
	 \]
	Its norm is computed as
	 \[
	 2|f| \| \tilde{q} \|^3 \cos  \theta_Z \| \tilde{q} - Z\| = 2 |f|  \frac{\cos  \theta_Z \tan \theta_Z}{\cos^3  \theta_{Z}} = 2 |f| \frac{\sin \theta_Z}{\cos^3 \theta_{Z}}.
	 \]
	By again restricting to a great circle passing trough the center $Z$, we get that this system has the force function $f \tan^2 \theta_Z$.
	}
	%{
	%We here make the similar projective argument as in the proof of Proposition \ref{prop: Kepler_Proj_Cor}.
	%The central projection from $\mathcal{S}_{SH}$ to $V$ gives the second time derivative equation (\ref{eq: sec_time_der}) after the time parameter change given by (\ref{eq:time_change_Kepler}).
	%We now consider the force field given by the force function $m \| \tilde{q} \|^2_a $ on $V$ as
	%\[
	%\ddot{\tilde{q}} = - 2 f   (\tilde{q} - Z).
	%\]	
	%After multiplying by $	\| \tilde{q} \|^3$ we obtain
	%\[
	%\| \tilde{q} \|^3 \ddot{\tilde{q}} = - 2 f \| \tilde{q} \|^3  (\tilde{q} - Z).
	%\]
	%Realize that its tangent component to $S_{SH}$ is given by
	 %\[
	 %- 2 f \| \tilde{q} \|^3 \cos \theta_Z (\tilde{q} - Z).
	 %\]
	%The norm of this component can be computed as
	%\[
	%|- 2 f| \| \tilde{q} \|^3 \cos \theta_Z  \cdot  \|\tilde{q} - Z\| = |-2f| \cos \theta_Z \cdot \tan \theta_Z  \cdot 
	%\| \tilde{q} \|^3  =  |-2f|   \sin \theta_Z| \cdot \| \tilde{q} \|^3 
	%\]
	%and it has the same norm with the force induced from the spherical Hooke problem. Indeed, one can see that
	%\[
	%\frac{d}{d \theta_Z}(-f \tan \theta_Z^2) = -2f \frac{\tan \theta_Z}{\cos^2 \theta_Z} =  - 2 f \sin \theta_Z \cdot  \| \tilde{q} \|^3.
	%\]
	%}
\end{proof}

\subsection{The Lagrange Problems in the plane and on the Sphere}

The Lagrange problem in the plane $\R^{2}$ is the system 
\begin{equation}\label{sys: planar Lagrange}
(\R^{2}, \|\cdot\|, m_{1} /\|q-Z_{1}\| + m_{2} /\|q-Z_{2}\|+f \|q-(Z_{1}+Z_{2})/2\|^{2}),
\end{equation}
with $m_1, m_2, f \in \R$, which is the superposition of two Kepler problems and a Hooke problem, with the Kepler centers placed symmetrically with respect to the Hooke center. 

Similarly, we define the Lagrange problem on the sphere as the system
\begin{equation}\label{sys: spherical Lagrange}
(\mathcal{S}, g_{st}, \hat{m}_{1} \cot \theta_{Z_{1}} + \hat{m}_{2} \cot \theta_{Z_{2}}+f \tan^{2} \theta_{Z_{mid}}),
\end{equation}
for which we assume that $Z_{2} \notin \{Z_{1}, -Z_{1}\}$. $\theta_{P}$ central angle of the moving particle to a point $P \in \mathcal{S}$, $Z_{mid}$ middle point of $Z_{1}$ and $Z_{2}$.

 Based on the previous Propositions {\ref{prop: Kepler_Proj_Cor} and \ref{prop: Hooke_Proj_Cor}}, we see that the following remarkable theorem holds
 
 \begin{theorem}\label{thm: Albouy Lagrange} {(Albouy \cite{Albouy1}) In the case $Z_{mid}$ is vertical, then the {spherical Lagrange problem on $\mathcal{S}_{SH}$} with masses $\hat{m}_{1}, \hat{m}_{2}, f$ is projected to a planar Lagrange problem in $V$, with the projections of the Kepler and Hooke centers as its own Kepler and Hooke centers, with the affine norm $\| \cdot \|_a$ and parameters  $m_{1}, m_{2}, f$ as determined by Proposition
% \marginpar{{Should we give a precise definition of the norm $\| \cdot \|_a$ here too? Anyway it's needed for the computation of the energy in the proof of next proposition.} {Yes good idea.}}
  \ref{prop: Kepler_Proj_Cor}. }
 \end{theorem}

\begin{proof}
	{We assume that the $Z_{mid}$ is vertical \emph{i.e.} $Z_{mid}= (0,0,-1)$.
	Additionally, for the normalization purpose, we set $Z_1 = (0,a,-1)$ and $Z_2 = (0,-a,-1)$. We then define the norm in $V$ as 
	\[
	\|q\|_a = \sqrt{x^2 + \frac{y^2}{1 + a^2}}
	\]
	for $q=(x,y,-1)$.}
{The affine norm $\| \cdot \|_a$ in $V$ was chosen}
%There is one affine norm in $V$ which can be chosen 
as common for all the three central force problems, two Kepler problems and a Hooke problem. By the principle of superposition, we may thus superpose them and the conclusion follows from the previous Propositions {\ref{prop: Kepler_Proj_Cor} and \ref{prop: Hooke_Proj_Cor}}. 
\end{proof}

{As a consequence to Theorem \ref{thm: Albouy Lagrange}, we have} 

\begin{prop} \label{prop: Albouy Lagrange} The energy of the spherical Lagrange problem induces an additional first integral for the planar Lagrange problem independent of its energy. Vice versa, the energy of the planar Lagrange problem induces an additional first integral for the spherical Lagrange problem independent of its energy.
\end{prop}

\begin{proof} 
	%Take care of the extension from the hemisphere to the sphere. 
	{The conservation of the energy of the planar problem in the spherical problem as well as the conservation of the energy of the spherical problem in the planar problem both follow from the fact that these systems are in correspondence, so their orbits in the configuration spaces are equivalent up to a time reparametrization.} %Writing these first integrals in a common chart as in  \cite{zhao2021} shows that they are independent.}  \marginpar{{It might be a good idea to show this independency here as well. }}}
	
	{To show their independence, we give their explicit expressions in a common chart as in \cite{zhao2021}.}
	{To normalize our situation, we here again assume that $Z_{mid}= (0,0,-1), \tilde{Z}_1= (0,a,-1)$, and $\tilde{Z}_2 = (0,-a,-1)$.}
	{Then the planer energy for the Lagrange problem in $V$ is described as
		\begin{align*}
		E_{pl} &=\frac{\| \dot{\tilde{q}} \|_a^2 }{2} - f\|\tilde{q}  \|^2_a  - \frac{m_1}{\| \tilde{q} - \tilde{Z}_1 \|_a } - \frac{m_2}{\| \tilde{q} - \tilde{Z}_2 \|_a }  \\
		&= \frac{1}{2} \left(\dot{\tilde{x}}^2 + \frac{\dot{\tilde{y}}^2}{1 + a^2}\right) - f \left(\tilde{x}^2 + \frac{\tilde{y}^2}{1+a^2}\right) - \frac{m_1}{\sqrt{\tilde{x}^2 + \frac{(\tilde{y} -a)^2}{1 + a^2}}}- \frac{m_2}{\sqrt{\tilde{x}^2 + \frac{(\tilde{y} +a)^2}{1 + a^2}}},
		\end{align*}}
	where $\tilde{q}=(\tilde{x}, \tilde{y}) \in V$ and $(\dot{\tilde{x}},\dot{\tilde{y}})\in T_{\tilde{q}}V$

		{We now write the energy of the spherical problem in the gnomonic chart $V$. In $\mathbb{R}^3$ the spherical kinetic energy is given by 
		\begin{equation}
		\label{eq: sp_kinetic_energy}
		 \frac{x'^2 + y'^2 + z'^2}{2},
		\end{equation}
		where $q=(x, y ,z) \in \mathcal{S}_{SH} $ and $(x',y',z')\in T_{q}\mathcal{S}_{SH}$. 
		Let $q  =(x, y, z)  \in \mathcal{S}_{SH}$ and $\tilde{q}= (\tilde{x}, \tilde{y} -1) \in V$ be corresponded via the central projection as 
		\[
		x= \frac{\tilde{x}}{\sqrt{\tilde{x}^2 + \tilde{y}^2 + 1}}, \quad y= \frac{\tilde{y}}{\sqrt{\tilde{x}^2 + \tilde{y}^2 + 1}}, \quad z= -\frac{1}{\sqrt{\tilde{x}^2 +\tilde{y}^2 + 1}}.
		\]
		Then the corresponding push-forward transformation from $ T_{\tilde{q}}V$ to $T_{q} \mathcal{S}_{SH}$ is given by
		\[
		\begin{pmatrix}
		x' \\
		y' \\
		z' 
		\end{pmatrix}
		= 
		\begin{pmatrix}
		\frac{\tilde{y}^2 +1}{(\tilde{x}^2 + \tilde{y}^2 + 1)^{3/2}} & -\frac{\tilde{x}\tilde{y}}{(\tilde{x}^2 + \tilde{y}^2 + 1)^{3/2}} \\
		-\frac{\tilde{x}\tilde{y}}{(\tilde{x}^2 + \tilde{y}^2 + 1)^{3/2}} & \frac{\tilde{x}^2 +1}{(\tilde{x}^2 + \tilde{y}^2 + 1)^{3/2}} \\
		\frac{\tilde{x}}{(\tilde{x}^2 + \tilde{y}^2 + 1)^{3/2}} & \frac{\tilde{y}}{(\tilde{x}^2 + \tilde{y}^2 + 1)^{3/2}}
		\end{pmatrix}
		\begin{pmatrix}
		\tilde{x}' \\ \tilde{y}' 
		\end{pmatrix}
		\]
		Using this, the projection of the spherical kinetic energy is represented as
		\[
		\frac{(1 + \tilde{y}^2)\tilde{x}'^2 - 2\tilde{x}\tilde{y} \tilde{x}' \tilde{y}' + (1 + \tilde{x}^2 )\tilde{y}'^2}{2(\tilde{x}^2 + \tilde{y}^2 +1)^2} 
		\]
		at $\tilde{q} = (\tilde{x},\tilde{y},-1 )\in V$.
		Remember that $(\cdot) '$ is the time derivative with respect to the time parameter $\tau$ defined as \eqref{eq: time_parameter_change}. 
		%Thus, this can be rewritten into 
		%\[
		%\frac{\dot{\tilde{x}}^2 + \dot{\tilde{y}}^2+ \dot{\tilde{z}}^2}{2\tilde{z}^4}.
		%\]
		%\marginpar{{Here we use $\dot{\tilde{x}}$} notation. How to avoid? }
		From this, the spherical kinetic energy in the gnomonic chart has an expression \cite{zhao2021}
		\[
		K_{sp}:= \frac{(1 + \tilde{y}^2)\dot{\tilde{x}}^2 - 2\tilde{x}\tilde{y} \dot{\tilde{x}} \dot{\tilde{y}} + (1 + \tilde{x}^2 )\dot{\tilde{y}}^2}{2} = \frac{\dot{\tilde{x}}^2+ \dot{\tilde{y}}^2+ (\dot{\tilde{x}}\tilde{y} - \tilde{x} \dot{\tilde{y}})^2}{2}
		%K_{sp}:=\frac{(1 + y^2)\dot{x}^2 - 2xy \dot{x} \dot{y} + (1 + x^2 )\dot{y}^2}{2(x^2 +y^2+1)}
		\]
		at $(\tilde{x},\tilde{y},-1)= (-x/z, -y/z,-1)$ in $V$, }%\marginpar{More details}
    		which can be seen as the combination of the planer kinetic energy and the squared angular momentum. %\marginpar{{One of the time derivative should be $'$ and the other should be $\cdot$. Please have another look.}}

	       {The spherical potential consists of the terms $-f\tan^2 \theta_{Z_{mid}}$, $-\hat{m}_1 \cot \theta_{Z_1}$, and $-\hat{m}_2 \cot \theta_{Z_2}$. They are expressed in the gnomonic chart $V$ as
		\[
		-f (\tilde{x}^2 + \tilde{y}^2),
		\]
		\[
		-\hat{m}_1 \frac{a \tilde{y} + 1}{\sqrt{(\tilde{y}-a)^2 + (1 + a^2)\tilde{x}^2}},
		\] 
		and 
		\[
		-\hat{m}_2 \frac{-a \tilde{y} + 1}{\sqrt{(\tilde{y}+a)^2 + (1 + a^2)\tilde{x}^2}},
		\]
		respectively. }
		
		{Combining these, we get the following expression of the spherical energy of the {Lagrange} problem} in the gnomonic chart:
		\begin{equation*}
		\begin{split}
		E_{sp}= &\frac{(1 + \tilde{y}^2)\dot{\tilde{x}}^2 - 2\tilde{x}\tilde{y} \dot{\tilde{x}} \dot{\tilde{y}} + (1 + \tilde{x}^2 )\dot{\tilde{y}}^2}{2} \\
		& -\hat{m}_1 \frac{a \tilde{y} + 1}{\sqrt{(\tilde{y}-a)^2 + (1 + a^2)\tilde{x}^2}}  -\hat{m}_2 \frac{-a \tilde{y} + 1}{\sqrt{(\tilde{y}+a)^2 + (1 + a^2)\tilde{x}^2}}- f (\tilde{x}^2 + \tilde{y}^2).
		\end{split}
		\end{equation*}
		%\marginpar{{Maybe it is better to write these computations outside of the proof?}{I think it is okay like this, or?}}
		{The functional independence of $E_{pl}$ and $E_{sp}$ now follows from these expressions.}%\marginpar{Compute Jacobian}
		{ Indeed one can check that the Jacobi matrix
		\[
		J :=
		\begin{pmatrix}
		\frac{d E_{pl}}{d\tilde{x}} & \frac{d E_{pl}}{d\tilde{y}} &  \frac{d E_{pl}}{d \dot{\tilde{x}}} &  \frac{d E_{pl}}{d \dot{\tilde{y}}} \\
		\frac{d E_{sp}}{d\tilde{x}} & \frac{d E_{sp}}{d\tilde{y}} &  \frac{d E_{sp}}{d \dot{\tilde{x}}} &  \frac{d E_{sp}}{d \dot{\tilde{y}}} 
		\end{pmatrix}
		\]
	has rank 2. To see this, it suffices to {observe that the} $2 \times 2$ {sub}matrix 
	\[
	\begin{pmatrix}
	 \frac{d E_{pl}}{d \dot{\tilde{x}}} &  \frac{d E_{pl}}{d \dot{\tilde{y}}} \\
	 \frac{d E_{sp}}{d \dot{\tilde{x}}} &  \frac{d E_{sp}}{d \dot{\tilde{y}}} 
	\end{pmatrix}
	= 
	\begin{pmatrix}
	\dot{\tilde{x}} &  \frac{\dot{\tilde{y}}}{1 + a^2} \\
	(\tilde{y}^2 +1)\dot{\tilde{x}}  -\tilde{x}\tilde{y} \dot{\tilde{y}}& -\tilde{x}\tilde{y} \dot{\tilde{x}} + (\tilde{x}^2 + 1)\dot{\tilde{y}}
	\end{pmatrix}
	\]
	has rank 2. 
}
		
	{Therefore we get an additional first integral for the planar problem from its corresponding spherical problem.} 	%{By applying Proposition 2.1 in \cite{zhao2021} to the correspondence between the planer and the spherical Lagrange problem, we obtain the conservation of both energies in each system.}
	
	{Similarly,  the same argument equips the spherical problem in $\mathcal{S}_{SH}$ with an additional first integral. }
	
	{We now show that the projected planar energy to $\mathcal{S}_{SH}$ extends to $\mathcal{S}$ in an analytical way, outside of its singularities, thus the integrability extends to the problem on $\mathcal{S}$.}

	%The left argument is the extendability of the projected planer energy to $\mathcal{S}$.
	{ We first consider the kinetic energy and we provide a differently, more direct argument as in \cite{zhao2021}. }  The {planar} kinetic energy at $\tilde{q}=(\tilde{x},\tilde{y},-1) $ on $V$ is given by   %\marginpar{{Again, there is a time change and it is necessary to express the expressions with the right time derivatives. Please have another look.} {More precisely, here it is maybe not needed to change time, but it has to be explained and we should always separate ``push-forward'' and ``projection''. The former implies no time change while the latter needs this time change.}}
	\begin{equation}
	\label{eq: pl_kinetic_energy}
	\frac{1}{2}\Bigl(\dot{\tilde{x}}^2 + \frac{\dot{\tilde{y}}{^{2}}}{1+a^2}\Bigr)
	\end{equation}
	for which we have taken the affine change of norm given by (\ref{eq: metric_a}) into account. {We now change the time parameter according to \eqref{eq: time_parameter_change}, then the above quantity can be rewritten into
	\[
	\frac{1}{2}\Bigl(\tilde{x}'^2 + \frac{\tilde{y}'^{2}}{1+a^2}\Bigr)(\tilde{x}^2 + \tilde{y}^2 + 1)^{-2}
	\]
	}
	{Let ${q=(x, y, z)\in \mathcal{S}_{SH}}$ be the centrally projected point of $\tilde{q}= (\tilde{x},\tilde{y},-1)\in V$ on $\mathcal{S}_{SH}$. We have }
		\[
		\tilde{x} = -\frac{x}{z},  \quad \tilde{y} = -\frac{y}{z}.
		\] 
	Then the push-forward transformation from $T_{q} \mathcal{S}_{SH}$ to $T_{\tilde{q}}V$  is given by 
	\[
	{\begin{pmatrix}
	\tilde{x}'\\
	\tilde{y}'
	\end{pmatrix}
	=
	\begin{pmatrix}
	- \frac{1}{z} & 0 & \frac{x}{z^2} \\
	0 & -\frac{1}{z} & \frac{y}{z^2}
	\end{pmatrix}
	\begin{pmatrix}
	x' \\ y' \\ z'
	\end{pmatrix}.}
	\]
	{Using this, we obtain the transformed expression of the planer kinetic energy \eqref{eq: pl_kinetic_energy} defined on $\mathcal{S}_{SH}$ given by 
	\begin{equation}
	\label{eq: pl_kinetic_S_SH}
		\frac{((a^2+1)x'^2+y'^2)z^2-2 z'((a^2+1)xx'+yy')z+ z'^2((a^2+1) x^2 +y^2)}{2(a^2+1)}.
	%\frac{((a^2+1)\dot{X}^2+\dot{Y}^2)X^2-2 \dot{Z}(X(a^2+1)\dot{X}+Y\dot{Y})Z+ \dot{Z}^2(a^2 X^2 + X^2 +Y^2)}{2Z^4(a^2+1)}.
	\end{equation}
	at $q = (x, y, z) \in \mathcal{S}_{SH}$.} 
	{Realize that this expression \eqref{eq: pl_kinetic_S_SH} of the planer kinetic energy can be analytically extended to the whole sphere $\mathcal{S}$.}
	 %\marginpar{{Can you write more details for this computation?}}

	For the potential 
	\[
	 {- \frac{m_1}{\sqrt{\tilde{x}^2 + \frac{(\tilde{y} -a)^2}{1 + a^2}}}- \frac{m_2}{\sqrt{\tilde{x}^2 + \frac{(\tilde{y} +a)^2}{1 + a^2}}}- f \left(\tilde{x}^2 + \frac{\tilde{y}^2}{1+a^2}\right)}
	\]
	of the planer Lagrange problem in $V$, just as in \cite{zhao2021} we apply the change of coordinates 
	\[
	{\tilde{x} = -\frac{x}{z}, \quad \tilde{y}= - \frac{y}{z}}
	\]
	which is derived from the central projection:$ V \ni (\tilde{x},\tilde{y}, -1) \mapsto  (x,y,z) \in \mathcal{S}_{SH} $, and obtain the projected representation
	\begin{equation}
	\label{eq: planer_potential_pr1}
	{- \dfrac{m_1 z}{\sqrt{\frac{a^2 x^2 - a^2 z^2 - 2 a y z + x^2 - y^2}{(a^2 + 1)}}} - \dfrac{m_2 z}{\sqrt{\frac{a^2 x^2 - a^2 z^2 + 2 a y z + x^2 - y^2}{(a^2 + 1)}}}-f \left(\frac{(a^2 +1)x^2 + y^2}{(a^2 + 1)z^2}\right) }
	\end{equation}
	defined on $\mathcal{S}_{SH}$. This quantity can be analytically extended to the whole unit sphere $\mathcal{S}$, outside its singularities, which are the  Kepler centers and their antipodal points and the horizontal equator $\{ (x,y,z) \in \mathcal{S} \mid z=0  \}$, when the corresponding mass parameter is not zero.
	%Mind that the projected planer potential given by \ref{eq: planer_potential_pr1} defined on $S$ is singular at south and north pole, 
	%and any points on the equator $\{ (X,Y,Z) \in S \mid Z=0  \}$.
	%}
	%{
	%By combining these arguments on extendability of the kinetic energy and the potential, we get the additional first integral defined on the whole sphere $\mathcal{S}_{SH}$ for the spherical Lagrange problem as well as the hemispherical case.

\end{proof}

\section{Integrable Lagrange Billiards}
\label{sec: integrable Lagrange billiards}

\subsection{Billiard Correspondence at Confocal Conic Sections}
In this subsection, we consider the problem of projective correspondence of a reflection wall $\tilde{B}$ in $V$ and its corresponding reflection wall $B$ in $\mathcal{S}_{SH}$. Recall that in this case a projective correspondence refers to the property that the laws of reflection in $V$ and on $\mathcal{S}_{SH}$ correspond to each other via the central projection. When this holds, then the billiard trajectories correspond to each other. This property does not hold for general reflection wall $\tilde{B} \subset V$. In this section we show that this nevertheless holds for any conic sections in $V$ centered at $(0, 0, -1)$, with respect to a  \emph{compatible} $\|\cdot\|_{*}$ in $V$, meaning that the $\|\cdot\|$-distance of the foci of the conic section, defined with respect to $\|\cdot\|_{*}$, equals $2 a$.

\begin{prop}\label{prop: reflection wall correspondence} {Any {centered} confocal conic section $\tilde{B} \subset V$ is projected to a {centered} confocal conic section $B \subset \mathcal{S}_{SH}$}. The foci of $B$ are the projection of the foci of $\tilde{B}$ by the central projection. The law of reflection at $\tilde{B}$ with respect to a compatible $\|\cdot\|_{*}$ and the law of reflection at $B \subset \mathcal{S}_{SH}$ correspond to each other.  \end{prop}
%\marginpar{{I think it is better to show the projective correspondence of conic sections and reflections separatedly. Otherwise the proof of this proposition might be too long.} {I tried to separate but then the second proposition will have to repeat everything the first proposition states, so what about keep it like this?}}

\begin{proof}
	{Since spherical Kepler problems in $\mathcal{S}_{SH}$ and planer Kepler problems are in correspondence as described in Proposition \ref{prop: Kepler_Proj_Cor}, their orbits are projected to each other up to some time parametrization. Any connected component of confocal conic sections in a plane/on a sphere is an orbit of the planer/spherical Kepler problem with the center at one of the foci. Indeed any confocal ellipse and branch of any confocal hyperbola are orbits of Kepler problems with positive mass-factor, and for hyperbolas, the other branch is obtained as an orbit of Kepler problem with negative mass-factor.  Each connected component of a confocal conic sections is projected to a connected component of a conic section with a focus at the projected center which is an orbit of the spherical/planer Kepler problem with the corresponding projected center. We now look at the other focus and its correspondence. For this purpose, we regard the same conic section as an orbit of the planer/spherical Kepler but with the center at the other focus. Then from the same projective argument, one can see that the other focus is also projected from the corresponding focus.}
	
	%{In oder to show the projective correspondence between the law of reflections, it suffices to see the conservations of the projected energies under elastic reflection.}

	{We will now check the projective correspondence between the laws of reflection at confocal conic sections in $V$ and on $\mathcal{S}_{SH}$. We first construct such reflection walls in $V$ and on $\mathcal{S}_{SH}$.}
	
	{For the normalization purpose, we set two foci $\tilde{Z}_1= (0,a,-1)$, and $\tilde{Z}_2 = (0,-a,-1)$ in $V$, then the norm $\| \cdot \|_a$ in $V$ should be chosen as \eqref{eq: metric_a}.}
	
	{We consider a centered elliptic cone given by 
	\begin{equation}
	\label{eq: elliptic_cone}
	\frac{x^2}{\tan^2 \alpha} + \frac{y^2}{\tan^2 \beta} - z^2=0,
	\end{equation}
	with $\alpha , \beta \in [0,\pi/2]$ such that 
	\begin{equation}
	\label{eq: metric_relation_sp}
	1+a^2 = \frac{\tan^2 \beta+1}{\tan^2 \alpha+1}.
	\end{equation}
	}
	{The intersection of $V$ and the cone \eqref{eq: elliptic_cone} gives a centered ellipse 
	\[
	\tilde{F}:=\frac{\tilde{x}^2}{\tan^2 \alpha} + \frac{\tilde{y}^2}{\tan^2 \beta} - 1=0
	\]
	defined in $V$.}
	{The foci $(0,c), (0,-c)$ of the ellipse $\tilde{F} = 0$ depends on the {involved norm, and is computed as}
	\[
	\frac{c^2}{1 + a^2} = \frac{\tan \beta^2}{1 + a^2} - \tan^2 \alpha \Leftrightarrow c^2 = a^2
	\]
	}
	{This means the foci are at two centers $\tilde{Z}_1$ and $\tilde{Z}_2$, thus the ellipse $\tilde{F} = 0$ is confocal. }
	
	{From the first and the second statement of this proposition, the intersection of this elliptic {cone} \eqref{eq: elliptic_cone} and $\mathcal{S}_{SH}$ is again a confocal ellipse on $\mathcal{S}_{SH}$ and is given by the equation
	\begin{equation}
	\label{eq: sp_confocal_ellipse}
	F:=\frac{x^2}{\sin^2 \alpha} + \frac{y^2}{\sin^2 \beta} -1=0.
	\end{equation}
	}
	{To see the projective correspondence of elastic reflections, we show that velocities before and after the reflection at $F=0$ on $\mathcal{S}_{SH}$ is projected to velocities before and after the reflection at $\tilde{F}=0$ in V.} {Unfortunately we have not found a geometrical way to see this. Here we provide a proof with a direct computation.}
	
	{Set 
	\[
	q:=(x,y,z)= \left(\sin \alpha \cos \theta,  \sin \beta \sin \theta, -  \sqrt{1-\sin^2 \alpha \cos^2 \theta - \sin^2 \beta \sin^2 \theta} \right)
	\]
	which lies in a confocal ellipse $F=0$ on $\mathcal{S}_{SH}$. }{The tangent vector to the ellipse at the point $q$ is given by 
	\[
	s:= \left(- \sin \alpha \sin \theta, \sin \beta \cos \theta,  -\frac{(\sin^2 \alpha -\sin^2 \beta)\sin \theta \cos \theta}{\sqrt{1- \sin^2 \alpha \cos^2 \theta - \sin^2 \beta \sin^2 \theta}} \right),
	\]
	and the normal vector is given by 
	\[
	n:= \left(\frac{\sin \alpha \cos \theta}{\tan^2 \alpha}, \frac{\sin \beta \sin \theta}{\tan^2 \beta},   \sqrt{1-\sin^2 \alpha \cos^2 \theta - \sin^2 \beta \sin^2 \theta}   \right).
	\]
	}
	{When the velocity vectors before the reflection at $q$ is given as
	\[
	v= k_1 \cdot s + k_2 \cdot n,
	\]
	where $k_1, k_2 \in \mathbb{R}$ are coefficients, then the reflected vector becomes as  %\marginpar{\rd{It should be enough to check that normal vector projects to normal vector. Please revise.}}
	\[
	w= k_1 \cdot s - k_2 \cdot n,
	\]
	}
	{Clearly, tangent vectors are projected to tangent vectors along the reflection walls. To see that $v$ and $w$ are projected to velocities before and after the elastic reflection at the corresponding point $\tilde{q}$ in $\tilde{F}=0$, we observe that it suffices to check that the normal vector $n$ is projected to the corresponding normal vector at $\tilde{q} \in V$ with respect to the corresponding metric on $V$, since then $v$ and $w$ are projected to vectors in $V$ having the same tangential component and opposite normal components.}

	{The point $q$ lying in $F=0$} is projected to the point 
	\[
	\tilde{q} := \left(-\frac{x}{z}, -\frac{y}{z},-1 \right)= { \left(\dfrac{\sin \alpha \cos \theta}{\sqrt{1-\sin^2 \alpha \cos^2 \theta - \sin^2 \beta \sin^2 \theta}}, \dfrac{\sin \beta \sin \theta}{\sqrt{1-\sin^2 \alpha \cos^2 \theta - \sin^2 \beta \sin^2 \theta}}, -1\right)}
	\]
	{lying in $\tilde{F}=0$}.
	
	{The corresponding  push-forward transformation from $T_{q} \mathcal{S}_{SH}$ to $T_{\tilde{q}} V$ is given by 
	\[
	\begin{pmatrix}
	\tilde{x}' \\
	\tilde{y}'
	\end{pmatrix}
	=
	\begin{pmatrix}
	\frac{1}{\sqrt{(\cos^2 \alpha  -\cos^2 \beta) \cos^2 \theta + \cos^2 \beta  }}& 0 &  \frac{\sin \alpha \cos \theta}{( \cos^2 \alpha - \cos^2 \beta )\cos^2 \theta  + \cos^2 \beta} \\
	0 &  \frac{1}{\sqrt{(\cos^2 \alpha  -\cos^2 \beta) \cos^2 \theta + \cos^2 \beta  }} & \frac{\sin \beta \cos \theta}{( \cos^2 \alpha - \cos^2 \beta )\cos^2 \theta  + \cos^2 \beta}
	\end{pmatrix}
	\begin{pmatrix}
	x'\\
	y'\\
	z'
	\end{pmatrix}
	\]
	}
	{Using this, the tangent vector $s$ is projected to the (tangent) vector 
	\begin{align*}
	\tilde{s} = &\frac{1}{( \cos^2 \alpha-\cos^2 \beta)\cos^2 \theta + \cos^2 \beta)^{3/2}}(-\sin \alpha \cos^2 \beta \sin \theta, \sin \beta \cos^2 \alpha \cos \theta) 
	\end{align*}
	and the normal vector $n$ is projected to the vector
	\begin{align*}
	\tilde{n} =  \frac{1}{\sin \alpha \sin \beta( \cos^2 \alpha-\cos^2 \beta )\cos^2 \theta + \cos^2 \beta)^{1/2}}(\sin \beta \cos \theta, \sin \alpha \sin \theta).
	\end{align*}
	We {ignore the factors and take}
	\begin{align*}
	&\hat{s} = (-\sin \alpha \cos^2 \beta \sin \theta, \sin \beta \cos^2 \alpha \cos \theta), \\
	&\hat{n} =  (\sin \beta \cos \theta, \sin \alpha \sin \theta).
	\end{align*}
	{Their inner product with respect to $\|\cdot\|_{a}$ is}
	\begin{equation*}
	\langle \hat{s},\hat{n} \rangle_a = - \sin \alpha \sin \beta \cos^2 \beta \sin \theta \cos \theta  + \frac{\tan^2 \alpha +1}{\tan^2 \beta + 1}\sin \alpha \sin \beta \cos^2 \alpha \sin\theta \cos \theta=0.
	\end{equation*}
	}
	{They are thus orthogonal. Hence, the projection $\tilde{n}$ of $n$ is indeed a normal vector at $\tilde{q} \in \{\tilde{F}=0 \}$ in $V$. }

	{Thus, the law of reflection at centered confocal ellipses in $V$ and the law of reflection at centered confocal ellipses on $\mathcal{S}_{SH}$ correspond to each other.}
	{The case of reflections at centered confocal hyperbolae is completely analogous.}
	%This fact means that the projected planer kinetic energy, thus the projected planer energy is preserved before and after the reflection against a confocal conic section reflection wall on $\mathcal{S}_{SH}$.

		%By combining the above arguments, we conclude that the reflection law at any confocal conic section in $V$ and the reflection law at  any confocal conic section on $\mathcal{S}_{SH}$ correspond to each other.
		
\end{proof}

\subsection{Integrability of Lagrange Billiards with Confocal Conic Section Reflection Walls}\label{Subsection: proof of thm 1}

We now prove Theorem \ref{thm: Theorem 1} for the planar and spherical problems.

\begin{proof} %By Theorem  \ref{thm: Albouy Lagrange} and Proposition \ref{prop: Albouy Lagrange}, for the planar and spherical Lagrange problems there exist additional first integrals, the ones induced by the spherical resp. planar energies, independent of the energies of their owns.

 %Moreover, these additional first integrals are invariant under reflections at the respective reflection walls according to Proposition \ref{prop: reflection wall correspondence}. The proof is now complete.
	
%{We are only left to show the conservation both energies under reflections at the respective reflection walls.}
 {From Proposition  \ref{prop: reflection wall correspondence}, we know the spherical and planer law of reflection at centered confocal conic sections are in correspondence, meaning that the incoming and the outgoing velocity vectors of an elastic reflection against such reflection walls in the plane are projected again to the incoming and outgoing velocity vectors of an elastic reflection against the corresponding reflection walls on the sphere, up to a time change which depends only on the point of reflection. Therefore the billiard trajectories on the sphere are projected to billiard trajectories in the plane in our situation, in which the underlying mechanical systems are in correspondence. As a consequence, the energy of the spherical system, written in the gnomonic chart $V$, is invariant under the reflections at a corresponding confocal conic section in $V$. Also, the energy of the planar system, while being expressed on $\mathcal{S}_{SH}$ and further extended to $\mathcal{S}$, is invariant under the reflections on $\mathcal{S}$ at a corresponding confocal conic section on the sphere. We get additional first integrals for both billiard systems independent of their energies. The proof is completed.}
%The conservation of the planer energy under spherical elastic reflections follows from the similar argument. 
\end{proof}

\subsection{Subcases of Integrable Lagrange Billiards}\label{Subsection: Subcases of Lagrange Billiards}
\subsubsection{The integrable free billiards}
The case $m_{1}=m_{2}=f=0$ of the system \eqref{sys: planar Lagrange}, and the case  $\hat{m}_{1}=\hat{m}_{2}=f=0$ of the system \eqref{sys: spherical Lagrange} correspond respectively to the cases of free motions in the plane and on the sphere. We recover the classical theorem of Birkhoff in the planar and spherical case.

\begin{cor} The free billiards in the plane and on the sphere with any combination of confocal conic section reflection walls are integrable.
\end{cor}

\subsubsection{The integrable Hooke billiards}
The case $m_{1}=m_{2}=0, f \neq 0$ of the system \eqref{sys: planar Lagrange}, and the case  $\hat{m}_{1}=\hat{m}_{2}=0, f\neq 0$ of the system \eqref{sys: spherical Lagrange} correspond respectively to the Hooke problems in the plane and on the sphere. In this case we recover the following theorem:
\begin{cor} The Hooke billiards in the plane and on the sphere with any combination of confocal conic section reflection walls centered at the Hooke center are integrable.
\end{cor}

\subsubsection{The integrable Kepler billiards}
The case $m_{1}=f=0, m_{2} \neq 0$ of the system \eqref{sys: planar Lagrange}, and the case  $\hat{m}_{1}=f=0, \hat{m}_{2}\neq 0$ of the system \eqref{sys: spherical Lagrange} correspond respectively to the Kepler problems in the plane and on the sphere. In this case we recover the following theorem:
\begin{cor} The Kepler billiards in the plane and on the sphere with any combination of confocal conic section reflection walls focused at the Kepler center are integrable.
\end{cor}

\subsubsection{The integrable Two-Center billiards}
The case $m_{1}, m_{2} \neq 0, f = 0$ of the system \eqref{sys: planar Lagrange}, and the case  $\hat{m}_{1}, \hat{m}_{2}\neq 0, f= 0$ of the system \eqref{sys: spherical Lagrange} correspond respectively to the two-center problems in the plane and on the sphere. In this case we recover the following theorem:
\begin{theorem} The billiards defined with the two-center problems in the plane and on the sphere with any combination of confocal conic section reflection walls focused at the two centers are integrable.
\end{theorem}

\subsubsection{The integrable billiards with superposition of Hooke and Kepler Problems}
The case $m_{1}, f \neq 0, m = 0$ of the system \eqref{sys: planar Lagrange}, and the case  $\hat{m}_{1}, f\neq 0, \hat{m}_{2}= 0$ of the system \eqref{sys: spherical Lagrange} correspond respectively to the superposition of a Hooke and a Kepler problems in the plane and on the sphere. In this case we recover the following theorem:
\begin{cor} The billiards defined with the superposition of a Hooke and a Kepler problems in the plane and on the sphere with any combination of confocal conic section reflection walls focused the Kepler center and centered at the Hooke center are integrable.
\end{cor}

\section{The Plane-Hyperboloid Projection and integrable Lagrange Billiards in the Hyberbolic Plane}
\label{sec: hyperbolic Lagrange billiards}

We now discuss the projection between the plane and the hyperbolic space, with the hyberboloid model for the latter.  
%The arguments are completely similar to the plane-spherical projection and thus will just be briefly indicated.

%The argument goes analogously to the plane-spherical case 
%and thus the presentation will be sketchy. 

%{The Model, }

%{the central projection}

%{Analogues of Albouy's theorem and related propositions (sketchy).}

%{Proof of Theorem \ref{thm: Theorem 1} for the hyperbolic case. }

%{Subcases in the hyperbolic plane. }

\subsection{The hyperboloid-Plane Projection}
{We consider the Minkowski space  $\mathbb{R}^{2,1}$, equipped with the pseudo-Riemannian metric
\begin{equation}
\label{eq: hyperbolid_metric}
dx^2 + dy^2 - dz^2.
\end{equation}
Consider the embedded two-sheeted hyperboloid given by the equation
$$\mathcal{H}:=\{(x,y,z) \in \mathbb{R}^{2,1}  \mid x^2 + y^2 -z^2 = -1   \}$$
and its {lower} sheet
$$\mathcal{H}_S:=\{(x,y,z) \in \mathcal{H} \mid z<0 \}.$$ } %\marginpar{It is prefered to use mathcal than mathbb for the surface}
 {The restriction of the pseudo-Riemannian metric $dx^2 + dy^2 - dz^2$ to $\mathcal{H}$ is Riemannian, and equipped both sheets of $\mathcal{H}$ with a \emph{hyperbolic metric}. {The space $\mathcal{H}_S$ equipped with this hyperbolic metric is called the hyperboloid model of the hyperbolic plane.}}

%We denote this two-sheeted hyperboloid by $\mathbb{H}:=\{(x,y,z) \in \mathbb{R}^{2,1}  \mid x^2 + y^2 -z^2 = -1   \}$ and the upper sheet by $\mathbb{H}_S:=\{(x,y,z) \in \mathbb{H} \mid z>0 \}$. 

{We consider the plane $V_{H}=\{ z= {-1} \} \subset \mathbb{R}^{2,1}$ which is tangent to $\mathcal{H}_{S}$ at its pole $(0,0,{-1})$.} The central projection from the origin of $ \mathbb{R}^{2,1}$ projects the lower sheet of hyperboloid $\mathcal{H}_S$ onto the unit disc $D:=\{(x,y) \in V_{{H}} \mid x^2 +y^2 {<} 1  \}$ {in} $V$, {and equips $D$ with an induced hyperbolic metric, making it the Klein disc model for the hyperbolic plane.}

{We denote by $\| \cdot \|_H$ the Minkowski {norm}  in  $\mathbb{R}^{2,1}$. 
Just as in the case of spherical-plane correspondence in Section \ref{Subsection: The Hemisphere-Plane Projection},} a force filed $F_{H}$ on $\mathcal{H}_{S}$ is carried to a force field $F_V$ on $V$ by the central projection. 

{Indeed, in this setting, a point $q \in \mathcal{H}_{S}$} is centrally projected to the point $\tilde{q} \in V_H$:
$$
q = \| \tilde{q} \|_H^{-1} \tilde{q}.
$$

{Suppose we have a natural mechanical system in $V_{H}$ with the equations of motion
$$\ddot{\tilde{q}}=F_{V}(\tilde{q}).$$
Thus we have
$$
\dot{q}  = \| \tilde{q} \|_H^{-2} (\dot{\tilde{q}}  \| \tilde{q} \|_H - \langle \nabla \| \tilde{q}\|_H , \dot{\tilde{q}} \rangle \tilde{q}).
$$
Again, we take a new time variable $\tau$ for the system on $\mathcal{H}_{S}$, and write $' :=\dfrac{d}{d\tau}$ so that 
\begin{equation}
\label{eq: time_parameter_change_hyp}
{\dfrac{d}{d \tau}= \| \tilde{q} \|_H^{2}\dfrac{d}{d t}},
\end{equation}
and consequently
$$
q'  =  (\dot{\tilde{q}}  \| \tilde{q} \|_H - \langle \nabla \| \tilde{q}\|_H , \dot{\tilde{q}} \rangle_H \tilde{q}),
$$
$$q''= \| \tilde{q} \|_H^{2} (\ddot{\tilde{q}}  \| \tilde{q} \|_H - (\langle \nabla \| \tilde{q}\|_H , \ddot{\tilde{q}} \rangle+\langle \overset{\cdot}{(\nabla \| \tilde{q}\|_H)} , \dot{\tilde{q}} \rangle) \tilde{q}).$$
We thus have
\begin{equation}\label{eq: second derivative correspondence_hyp}
q''=\| \tilde{q} \|_{H}^{2} (F_{V}(\tilde{q})  \| \tilde{q} \|_{H}- \lambda(\tilde{q}, \dot{\tilde{q}}, \ddot{\tilde{q}}) \tilde{q})
\end{equation}
in which we have set $\lambda(\tilde{q}, \dot{\tilde{q}}, \ddot{\tilde{q}})=\langle \nabla \| \tilde{q}\|_H , \ddot{\tilde{q}} \rangle+\langle \overset{\cdot}{(\nabla \| \tilde{q}\|_H)} , \dot{\tilde{q}} \rangle$.} {The gradient and the inner product are defined with respect to the pseudo-Riemannian metric \eqref{eq: hyperbolid_metric}.}%\marginpar{So here just normal gradient and normal euclidean inner product, or rather the minkowski ones? Please check.}

{The Levi-Civita connection of a pseudo-Riemannian manifold projects to the Levi-Civita connection of its embedded submanifold. In our case, $\mathcal{H}_{S}$ is Riemannian with the induced metric from $\R^{2,1}$. So again by projecting both sides of this equation to the tangent space $T_{q}  \mathcal{H}_{S}$, we get the equations of motion of the form
$$\nabla_{q'} q'=F_{H}(q).$$
}

{We see that to switch from the plane-sphere correspondence as in Section \ref{Subsection: The Hemisphere-Plane Projection} to the plane-hyperboloid correspondence with our setting, it is enough to properly change some signs in proper places while the others are completely similar. We shall make use of this similarity in the sequel to omit certain details.}

%We may change the time parameter according to $\dfrac{d }{d  \tau} = \| q_1 \|^2 \dfrac{d}{dt}$, and take the corresponding time derivative
%\[
%q'_0 = \dot{q}_1  \|q_1\| - \langle \nabla \| q_1\| , \dot{q}_1 \rangle q_1.
%\]

%For our purpose, we would like to have natural mechanical systems which are centrally projected to natural mechanical systems, \emph{i.e.} the question is, when we start from a natural mechanical system on $S_{SH}$ resp. $V$, then whether the projected system on $V$ resp. $S_{SH}$ is also derived from a potential and is thus also a natural mechanical system. As we would expect this does not hold in general. Nevertheless, it actually holds for some important systems.

\subsection{Projective Properties of the Hooke and Kepler Problems in the Hyperbolic plane}
%We consider a central force problem $F_{S}$ on $\mathcal{S}$ with a distinguished center $Z \in \mathcal{S}_{SH}$. By assumption the force field is invariant under the $SO(2)$-action by rotations around $Z$ on $\mathcal{S}$. The projected force field $\tilde{F}_{V}$ on $V$ is in general not derived from a potential. In the same way, a central force problem $\tilde{F}_{V}$ in $V$ with center $\tilde{Z}$ might not project to a system derived from a potential on $\mathcal{S}_{SH}$.

%There are special cases that this does hold. The first is relatively easy to see: when $Z=(0,0, -1)$, the projected force field $\tilde{F}_{V}$ is also invariant under the  $SO(2)$-action by rotations in $V$ as inherited from rotations around the vertical axis in $\R^{3}$, and therefore $\tilde{F}_{V}$ is derived from a potential. The second case is maybe not as easy to see: The point $Z \in \mathcal{S}_{SH}$ can be chosen arbitrary, and $\tilde{F}_{V}$ will be derived from a potential when $F_{S}$ is the force field of the Kepler-Serret Problem on the sphere, and in this case $\tilde{F}_{V}$ itself is the force field of a Kepler problem in $V$ for a proper choice of an affine metric. Also, among the problems belonging to the first case, the Hooke problems has the property that $\tilde{F}_{V}$ is derived from a potential for \emph{any} affine metrics in $V$. 

\subsubsection{The Kepler Problems and The Hooke Problems}

We first discuss the case of the Kepler problems. The hyperbolic Kepler problem is the natural mechanical system $(\mathcal{H}_{S}, g_{H}, \hat{m} \coth \theta_{Z})$, in which $g_{H}$ is the induced hyperbolic metric on $\mathcal{H}_{S}$, $\hat{m} \in \R$ is the mass-factor and the angle $\theta_{Z}$ is the central hyperbolic angle the moving particle made with  the center $Z \in \mathcal{H}_{S}$. 
%The system naturally restricts to a natural mechanical system $(\mathbb{H}_{S}, g_{H}, \hat{m} \coth \theta_{Z})$ by restriction.
%In the case that $Z$ is \emph{vertical}, $Z=(0,0, 1)$, it is not hard to see by a direct computation that the hyperbolic Kepler-problem is projected to the planar Kepler problem $(V, \|\cdot\|, \hat{m}/\|\tilde{q}\|)$. Consequently the orbits of the spherical Kepler problem are all conic sections on the sphere by means of orbital correspondence and analytic extension. A special property of the Kepler problem is that this remains true when $Z$ is not vertical, up to a change of metric and of the mass factor CITATIONS. 

{Without loss of generality, we set $Z=  \left(0, \frac{a}{\sqrt{1 - a^2}}, -\frac{1}{\sqrt{1 - a^2}}\right) \in \mathcal{H}_{S}$ for $a \in (-1, 1)$, and $\tilde{Z}=(0, a, -1)$ the projection point of $Z$ in $D \subset V_{H}$.} For $\tilde{q}=(\tilde{x}, \tilde{y}, -1) \in V_H$ we define
\begin{equation}
\label{eq: metric_a_hyp}
\| \tilde{q}\|_{{a}}= \sqrt{ \tilde{x}^2 + \frac{\tilde{y}^2}{1-a^2}}.
\end{equation}
%which is an affine change of metric from the induced metric on $V$ of the standard Euclidean metric $\|\cdot\|$ in $\R^{2}$.
%{This situation is explained in the following proposition and this fact has been originally shown in \cite{zhao2021}} \marginpar{{I think it should be attribute to Higgs or Halphen}}

{Similar to the case of plane-spherical correspondence, we have}

\begin{prop} 
	\label{prop: Kepler_Proj_Cor_hyp}
	{The hyperbolic Kepler problem {$(\mathcal{H}_{S}, g_{H}, \hat{m} \coth \theta_{Z})$} projects to the planar Kepler problem $(V_{H}, \|\cdot \|_{a}, m/\|\tilde{q}-\tilde{Z}\|_{a})$ such that {$m = \dfrac{\hat{m}}{\sqrt{1 - a^2}}$.}}
\end{prop} 

By analyticity, a proof of this proposition follows from Proposition \ref{prop: Kepler_Proj_Cor} by formally substituting $(x, y, a, z)$ by $(i x, i y, i a, z)$ and argue with the equations of motion. The geometric proof of Proposition \ref{prop: Kepler_Proj_Cor} also carries over to this hyperbolic case, but now using hyperbolic geometry.
%\begin{proof}
%	By using the analyticity, we can extend the spherical Kepler problem $(\mathcal{S}_{SH},g_{st}, \hat{m} \cot \theta_{Z})$ and the Kepler problem $(V, \| \cdot \|_a, m/\| \tilde{q} -\tilde{z}\|_a)$ to a complex space $\mathbb{C}^3$ and get the projective correspondence between complexified systems. After restricting these systems to a pure imaginary space $(i \mathbb{R})^3$ one get the desired correspondence between the hyperbolic Kepler problem $(\mathbb{H}_{S}, g_{H}, \hat{m} \coth \theta_{Z})$ and the planer Kepler problem $(V_{H}, \|\cdot \|_{a}, m/\|\tilde{q}-\tilde{Z}\|_{a})$. This can be easily seen by substituting $(ix,iy,z)$ to $(x,y,z)$ and $ia$ to $a$ in Proposition \ref{prop: Kepler_Proj_Cor}.
%\end{proof}

%\subsection{The Hooke Problems}
Our second case is the Hooke problems. The hyperbolic Hooke problem is the natural mechanical systems given by $(\mathcal{H}_{S}, g_{H}, f \tanh^{2} \theta_{Z})$ with the {mass}-factor $f \in \R$.

Analogously as in the Kepler case, we get the following correspondences between Hooke systems.

\begin{prop} 
	\label{prop: Hooke_Proj_Cor_hyp}
	The hyperbolic Hooke problem $(\mathcal{H}_{S}, g_{H}, f \tanh^{2} \theta_{Z})$ with $Z=(0,0,-1)$ projects to any of the Hooke problems in $V$ of the form $(V_H, \|\cdot|\|_{a}, f \|\tilde{q}\|_{a}^{2})$ for any $a \in \R$. 
\end{prop}
In contrast to the Kepler case, we can freely choose the parameter $a$ in the affine changed norm $\| \cdot \|_a$ for the Hooke problems. 
%This correspondence can be shown by using the same extension argument as in Proposition \ref{prop: Kepler_Proj_Cor_hyp}.

\subsection{The Lagrange Problems in the Plane and in the Hyperbolic Plane} 
By superposing two hyperbolic Kepler problems and a hyperbolic Hooke problem, we obtain the hyperbolic Lagrange problem 
\[
(\mathcal{H}, g_{H},  \hat{m}_1 \coth \theta_{Z_1}+ \hat{m}_2 \coth \theta_{Z_2}+f \tanh^{2} \theta_{Z_{mid}}),
\]
%on any sheet $\mathcal{H}$ of two-sheeted hyperboloid
for which we assume that $Z_1$ and $Z_2$ are in the same sheet of two-sheeted hyperboloid $\mathcal{H}$. 
Here, $\theta_{P}$ is a hyperbolic central angle of the moving particle to a point $P \in \mathcal{H}$
%\marginpar{Does it make sanse to extend the system to two sheets of hyperboloid? A particle cannot jump from one sheet to the other sheet anyway.}

By combining the previous Propositions \ref{prop: Kepler_Proj_Cor_hyp} and \ref{prop: Hooke_Proj_Cor_hyp}, we get the following correspondence on the Lagrange problems in the plane and in the hyperbolic plane as an analogy of the spherical case. 

\begin{theorem}
	\label{thm: Albouy Lagrange_hyp}
	In the case $Z_{mid}$ is vertical, then the hyperbolic Lagrange problem on $\mathcal{S}_{H}$ with masses $\hat{m}_1, \hat{m}_2,f \in \R$ is projected to the planer Lagrange problem in $V_{H}$, with the projections of the Kepler and the Hooke centers as its own Kepler and Hooke centers, with the affine norm $\| \cdot \|_a$ and parameters $m_1,m_2,f$ as determined by Proposition \ref{prop: Kepler_Proj_Cor_hyp}.
\end{theorem}

From this theorem, we get the following proposition as a consequence.
\begin{prop}
	The energy of the hyperbolic Lagrange problem induces an additional first integral for the planer Lagrange problem independent of its energy. Vice versa, the energy of the planer Lagrange problem induces an additional first integral for the hyperbolic Lagrange problem independent of its energy.
\end{prop}

\subsection{Integrable Lagrange Billiards in the Hyperbolic Plane }
We here consider the presence of a confocal conic section reflection wall $\tilde{B}$ in $V_H$ and its corresponding reflection wall $B$ in $\mathcal{H}_S$. The proof goes analogously as in the case of plane-spherical correspondence. 
%In the following proposition, we will see the projective correspondence between the law of reflection in $V$ at $\tilde{B}$ and the law of reflection in $\mathcal{H}_S$ at $B$, when $\tilde{B}$ is any confocal conic section.

\begin{prop}
	Any confocal conic section $\tilde{B} \subset V_{H}$ is projected to a confocal conic section $B \subset \mathcal{H}_S$. The foci of $B$ are the projection of the foci of $\tilde{B}$ by the central projection. The law of reflection at $\tilde{B}$ with respect to a compatible $\| \cdot \|_*$ and the law of reflection at $B \subset \mathcal{H}_S$ correspond each other.
\end{prop}

%\begin{proof}
%	The first and the second statement follows from the same argument as in the proof of Proposition REF. 
%\end{proof}

\subsection{Proof of Theorem \ref{thm: Theorem 1} in the Hyperbolic Case and the Subcases}
With all these ingredients, the proof of Theorem \ref{thm: Theorem 1} for the spherical and planar case from Section \ref{Subsection: proof of thm 1} carries directly to the hyperbolic case as well, which completes the proof of Theorem \ref{thm: Theorem 1} in all cases.

Also, the subcases as listed in Section \ref{Subsection: Subcases of Lagrange Billiards} carries to integrable systems defined on the hyperbolic plane as well.
	
\section{The Complex Square Mapping and Hooke-Kepler Correspondence in the Hyperbolic Space and on the Sphere}
\label{sec: Hooke-Kepler conformal correspondence }
%\marginpar{\rd{We have chosen the other sheet of the hyperbola. Please make the changes accordingly.}
%\bl{I think nothing will be changed here. The minus sign in z-coodinates is canceled out in kinetic energy and potentials.}}
The classical conformal correspondence between the planar Hooke and Kepler problems via the complex square mapping has been generalized to conformal correspondences among the Hooke problems defined on the sphere, in the hyperbolic plane, and the Kepler problem defined in the hyperbolic plane by Nersessian and Pogosyan \cite{Nersessian- Pogosyan}. We explain that these conformal correspondences extend to integrable billiards defined with these natural mechanical systems. 

We take the plane ${\{z=0 \}}$ as a stereographic chart from the North pole $(0,0,1)$ of the unit sphere $\mathcal{S}$. For the hyperbolic plane we take the Poincar\'e disc model in the unit disc in the plane ${\{z=0 \}}$, seen as projection of the hyperboloid model from the {"North pole"} $(0,0,1)$. We identify {the plane ${\{z=0 \}}$} with $\C$ in which the Poincar\'e disc is $\mathcal{D}:=\{w \in \C \mid |w|<1\}$.

The round metric on $\mathcal{S}$ is represented in the stereographic chart  as
 \begin{equation}
 \label{eq: stereo_sph_metric}
 \frac{4}{(1 + |q|^2)^2} dq d\bar{q}.
 \end{equation}
 Analogously the Poincar\'e disk $\mathcal{D}$ is equipped with the hyperbolic metric
 \begin{equation}
 \label{eq: stereo_ps_sph_metric}
 \frac{4}{(1 - |q|^2)^2} dq d\bar{q}.
 \end{equation}

% From last section, we already know that the spherical and the pseudo-spherical Kepler billiards with focused conic section reflection walls and Hooke billiards with centered conic section reflection walls are integrable. We here apply conformal technique to these integrable mechanical billiards on constant curvatures and derive the correspondence between them. 
 
 The spherical kinetic energy in the stereographic chart is thus
 \[
 \frac{(1+ |q|^2)^2 |p|^2}{8}
 \]
 by using the cometric of (\ref{eq: stereo_sph_metric}). 
 In this stereographic chart, the force functions %\marginpar{\bl{Potential or force function?}} 
 of the spherical Hooke and spherical Kepler problems are given respectively as 
 \[
 -\frac{4 f | q |^2}{(1- | q|^2)^2}
 \]
 and 
 \[
 \hat{m} \frac{1- |q|^2}{2 |q|},
 \]
 respectively, with $f, \hat{m} \in \R$. 
 Analogously, the hyperbolic kinetic energy in the Poincare disk $\mathcal{D}$ is
 \[
 \frac{(1- |q|^2)^2 |p|^2}{8},
 \]
 with the force functions of the hyperbolic Hooke and hyperbolic Kepler problems
 \[
-\frac{4 f | q|^2}{(1+ | q|^2)^2}
\]
and 
\[
\hat{m} \frac{1+ |q|^2}{2 |q|},
\]
 respectively, with $f, \hat{m} \in \R$. 
 
\begin{prop} (Nersessian-Pogosyan \cite{Nersessian- Pogosyan}) \label{prop: complex square}
 	The spherical Hooke problem, the {hyperbolic} Hooke problem, and the {hyperbolic} Kepler problem are mutually in conformal correspondence. 
\end{prop}
\begin{proof}
	We start with the Hamiltonian of the spherical/hyperbolic Hooke problem
	\[
	\frac{(1\pm |z|^2)^2 |w|^2}{8} +  \frac{4 f | z|^2}{(1\mp | z|^2)^2} -\hat{m} = 0
	\]
	restricted to its $\hat{m}$-energy hypersurface. The signs determine whether it is the spherical or the hyperbolic problem we are considering. 
	By multiplying both sides by $\dfrac{(1\mp | z|^2)^2}{ | z|^2}$, we get 
	\[
	\frac{(1- |z|^4)^2 |w|^2}{8 | z|^2} + 4 f - \hat{m} \dfrac{(1\mp | z|^2)^2}{ | z|^2}= 0
	\]
	We now apply the conformal transformation $(z, w) \mapsto (z^2, w/2 \bar{z}):= (p,q)$ and the transformed Hamiltonian becomes
	\[
	\frac{(1- |q|^2)^2 |p|^2}{8 } +  4 f - \hat{m} \dfrac{(1\mp | q|)^2}{ | q|}= 0.
	\]
	after a proper time change. As we can rewrite this system into 
	\[
	\frac{(1- |q|^2)^2 |p|^2}{8 } +  4 f - \hat{m} \dfrac{1+ | q|^2}{ | q|} \pm 2 \hat{m} = 0,
	\]
	this is the Hamiltonian of a hyperbolic Kepler problem restricted to the energy level with energy $-(4 f \pm 2 \hat{m} )$. 
	
	{The same trick, with a multiplicative factor of $\dfrac{(1\mp | q|^2)^2}{(1\pm | q|^2)^2} $ gives a transformation between  the spherical and hyperbolic Hooke problems restricted to energy levels.}

%	Equally, we may also start with the Hamiltonian of the Kepler problem on a positive/negative curvature on the $f$-energy hypersurface
%	\[
%	\frac{(1\pm \|q\|^2)^2 \|p\|^2}{8} +  \frac{4s \| q\|^2}{(1\mp \| q\|^2)^2} \pm s +f \mp s = 0,
%	\]
%	which can be rewritten into
%	\[
%	\frac{(1\pm \|q\|^2)^2 \|p\|^2}{8} \pm s\frac{(1\pm \| q\|^2)^2}{(1\mp \| q\|^2)^2} +f \mp s = 0.
%	\]
%	By multiplying $\dfrac{(1\mp \| q\|^2)^2}{(1\pm \| q\|^2)^2} $, we get 
%	\[
%	\frac{(1 \mp \|q\|^2 )^2 \| p\|^2}{8} \pm s  + (f\mp s) \frac{(1\mp \| q\|^2)^2}{(1\pm \| q\|^2)^2}=0.
%	\]
%	This above system is equivalent to  
%	\[
%	\frac{(1 \mp \|q\|^2 )^2 \| p\|^2}{8} + f  \mp  4(f \mp s) \frac{\| q\|^2}{(1\pm \| q\|^2)^2}=0,
%	\]
%	which is the Hooke problem on a negative/positive curvature.
\end{proof}

\begin{cor} In the Poincar\'e disc in {the plane ${\{z=0 \}}\cong \C$}, the curve representing a branch of a conic section on the hyperboloid model focused at the ``South pole'' $(0,0,-1)$ is transformed via the complex square mapping $:\C \to \C: z \mapsto z^2$ into a curve simultaneously representing a conic section centered at the ``South pole'' on the hyperboloid model, and part of a conic section
%which is stereographically projected into the unit disc in $V$, 
{defined on the hemisphere $\mathcal{S}_{SH}$ centered at the South pole.}
\end{cor}

\begin{proof} This follows from Proposition \ref{prop: complex square}, which implies that an orbit of the hyperbolic Kepler problem is sent to an orbit of the spherical/hyperbolic Hooke problem up to a time parametrization. Thus the conclusion of the corollary follows.
\end{proof}

{We now show that any confocal family of centered spherical/hyperbolic conic sections is transformed into a confocal family of focused hyperbolic conic sections by {this} series of conformal transformations.

\begin{prop}
\label{prop: stereographic_confocal}
A family of confocal focused hyperbolic conic sections {on $\mathcal{H}_{S}$}, expressed in the Poincar\'e disc $\mathcal{D}$ are transformed {into} a family of confocal centered spherical/hyperbolic conic sections in the stereographic chart/Poincar\'e disc in the plane ${\{z=0 \}} \cong \C$ via the complex square mapping $\C \to \C: z \mapsto z^2$.
\end{prop}
\begin{proof}
	{We start with a family of confocal focused hyperbolic conic section on $\mathcal{H}_{S}$. Choose a  parameter $0 <a <1$, and suppose that a family of such hyperbolic conic sections has common centers at $(0, \frac{a}{\sqrt{1-a^2}}, -\frac{1}{\sqrt{1-a^2} })$.}
	%\rd{We take the parameter $0 <a <1$.} 
	%\marginpar{\rd{Please explain the meaning of $a$.}}

	We {take a new set of }orthogonal coordinates in the Minkowski space $\mathbb{R}^{2,1}$ as 
	\[
	u =x, \quad v = \frac{y+ az}{\sqrt{1-a^2}}, \quad w = \frac{ay + z}{\sqrt{1-a^2}}, 
	\]
	The pseudo-Riemannian metric defined by \eqref{eq: hyperbolid_metric} is {expressed in these new coordinates as}
	\[
	du^2 + dv^2 -dw^2.
	\]
	In this coordinates, the two-sheeted hyperboloid is given by the equation 
	\[
	\mathcal{H}= \{ (u,v,w)  \in \mathbb{R}^{2,1} \mid u^2 +v^2 -w^2= -1\}. 
	%=  \{ (x,y,z)  \in \mathbb{R}^{2,1} \mid x^2 +y^2 -z^2= -1\}.
	\]
	The plane $\hat{V}:=\{w=-1\}$ is tangent to the hyperboloid $\mathcal{H}$ at the the point $(u,v,w)= (0,0,-1)$. We equip $\hat{V}$ with the norm $\| \cdot \|_a$ defined as 
	\[
	\| (\tilde{u},\tilde{v}) \|^2_a = \tilde{u}^2 + \frac{\tilde{v}^2}{1-a^2}
	\]
	for $(\tilde{u}, \tilde{v}) \in \hat{V}$.
	We now consider {the family of} confocal centered ellipses in $\hat{V}$ with {foci} at $(0,-a,-1)$ and $(0,a,-1)$ ({with respect to $\| \cdot \|_a$}) given by the equation
	\begin{equation}
	\label{eq: confocal_conics_T}
	\frac{\tilde{u}^2}{\frac{B^2-a^2}{1-a^2}} + \frac{\tilde{v}^2}{B^2} - 1=0,
	\end{equation}
	where {$B>a$} is a positive parameter. 
	%Mind that the position of foci depends on the metric. 
	
	We now project {this family of} confocal centered ellipses in $\hat{V}$ to the hyperboloid by the central projection. Let $(u, v, w) \in \mathcal{H}_{S}$ be the centrally projected point of $(\tilde{u},\tilde{v},-1)$. Then we have 
	\[
	\tilde{u}= -\frac{u}{w}, \quad \tilde{v} = -\frac{v}{w},
	\]
	{and the transformed expression of the family of confocal ellipses is given by }
	\begin{equation}
	\label{eq: confocal_conics_H}
	\frac{u^2(1-a^2)}{w^2(B^2-a^2)} + \frac{v^2}{w^2B^2}-1=0.
	\end{equation}
	{As an implication of the }projective correspondence of the hyperbolic Kepler problem and the planer Kepler problem, the central projection {projects the} hyperbolic conic sections {to} conic sections in the plane and {projects foci to foci} when they are centered. Thus, the projected conic sections on {$\mathcal{H}_{S}$} is again confocal. 
	%\rd{Also,} the projected confocal conic sections have a focus at the ``South pole'' on the hyperboloid which is the central projection of the focus $(0,-a,-1)$ of the planer confocal conic sections \eqref{eq: confocal_conics_T}. \marginpar{\rd{I find it a bit hard to understand this last sentence. What about just leave it out?}}
	
	In the original coordinates $(x,y,z)$ in $\mathbb{R}^{2,1}$, the equation \eqref{eq: confocal_conics_H} can be written {as}
	\[
	\frac{x^2(1-a^2)^2}{(ay+z)^2(B^2-a^2)} + \frac{(y+az)^2}{(ay + z)^2 B^2}-1= 0.
	\]
	We now rewrite this in the {coordinates} $(q_1,q_2)$ in the Poincar\'e disc {$\mathcal{D}$} {with the stereographic projection} 
	\[
	x = \frac{2 q_1}{1- q_1^2 - q_2^2}, \quad y= \frac{2 q_2}{1- q_1^2 - q_2^2}, \quad z= -\frac{1+ q_1^2 + q_2^2}{1- q_1^2 - q_2^2}, 
	\]
	{which transforms the equation of the} confocal focused hyperbolic conic sections in the {Poincar\'e disc $\mathcal{D}$} {into} 
	\[
	\frac{4(1-a^2)q_1^2}{(B^2-a^2)(2a q_2- q_1^2 - q_2^2 - 1)^2} + \frac{(-2 q_2 + a(q_1^2 + q_2^2 + 1))^2}{B^2 (2 a q_2 - q_1^2 - q_2^2 - 1)^2} - 1 = 0.
	\]
	We now apply the complex square mapping. Set 
	\[
	{q_1+i q_{2} = (z_{1}+i z_{2})^{2}.}
	\]
	{and the above equation is now }
	\[
	\frac{4( 1- a^2)^2 (z_1^2 - z_2^2)^2}{(B^2 - a^2)(-z_1^4 - 2 z_1^2 z_2^2 - z_2^4 + 4 a z_1 z_2 - 1)^2} + \frac{-4 z_1 z_2 + a (z_1^4 + 2 z_1^2 z_2^2 + z_2^4 + 1))^2}{B^2(-z_1^4 - 2 z_1^2 z_2^2 - z_2^4 + 4 a z_1 z_2 - 1)^2} -1 =0.
	\]
	{Suppose that $(z_1,z_2) {\in \mathcal{D}}$ corresponds to the point $(x,y,z) \in \mathcal{S}_{SH}$ via stereographic projection: }
	\[
	z_1 = - \frac{x}{z}, z_2 = -\frac{y}{z}.
	\]
	{Then the above equation can be equivalently written as}
	\begin{equation}
	\label{eq: confocal_central_conics_S}
	\begin{split}
	& \frac{4(1-a^2)^2 (-1 + z)^4 (x^2 - y^2)^2}{(B + a) (B - a) (z^4 - 4 z^3 + (-4 xya + 6)z^2 + (8 xya - 4)z + x^4 + 2x^2 y^2 + y^4 - 4 a xy + 1)^2} \\
	& + \frac{(z^4 a - 4 z^3 a + (-4 xy + 6a) z^2 + (8 xy - 4a)z + (x^4 + 2x^2 y^2 + y^4 + 1)a - 4 xy)^2}{B^2  (z^4 - 4 z^3 + (-4 xya + 6)z^2 + (8 xya - 4)z + x^4 + 2x^2 y^2 + y^4 - 4 a xy + 1)^2} -1=0.
	\end{split}
	\end{equation}
	In order to see that this equation determines spherical conic sections with common centers at the ``South pole'' and common foci, we project them to the plane $V=\{ z=-1 \}$ by the central projection {and examine their images therein. }
	{In the gnomonic chart $V$, the above equation is expressed with coordinates $(\tilde{x}, \tilde{y}, -1) \in V$ as}
	\[
	x= \frac{\tilde{x}}{\sqrt{\tilde{x}^2 + \tilde{y}^2 + 1}}, \quad  y= \frac{\tilde{y}}{\sqrt{\tilde{x}^2 + \tilde{y}^2 + 1}}, \quad z=- \frac{1}{\sqrt{\tilde{x}^2 + \tilde{y}^2 + 1}}.
	\]
	By using Maple, this can be factorized into
	\begin{equation*}
	\begin{split}
	&4 ((4 \tilde{x}^2 + 4 \tilde{y}^2 + 8) \sqrt{\tilde{x}^2 + \tilde{y}^2 + 1} + \tilde{x}^4 + (2 \tilde{y}^2 + 8) \tilde{x}^2 + \tilde{y}^4 + 8 \tilde{y}^2 + 8) (1+\tilde{x}^2 + \tilde{y}^2 )\\
	&\times \left(-\frac{(-a^2 + B)(B + 1)\tilde{x}^2}{2} + a\tilde{y}(B - 1)(B + 1)\tilde{x} - \frac{(-a^2 + B)(B + 1)\tilde{y}^2}{2} - B^2 + a^2\right)\\
	&\times\left(-\frac{(a^2 + B)(B - 1)\tilde{x}^2}{2} + a\tilde{y}(B - 1)(B + 1)\tilde{x} - \frac{(a^2 + B)(B + 1)\tilde{y}^2}{2} - B^2 + a^2\right)=0
	\end{split}
	\end{equation*}

	The factors in the first line only takes positive value. Thus, we only consider the last two factors:
	\[
	G_1:=-\frac{(-a^2 + B)(B + 1)\tilde{x}^2}{2} + a\tilde{y}(B - 1)(B + 1)\tilde{x} - \frac{(-a^2 + B)(B + 1)\tilde{y}^2}{2} - B^2 + a^2
	\]
	and 
	\[
	G_2:=-\frac{(a^2 + B)(B - 1)\tilde{x}^2}{2} + a\tilde{y}(B - 1)(B + 1)\tilde{x} - \frac{(a^2 + B)(B + 1)\tilde{y}^2}{2} - B^2 + a^2.
	\]
	In the rotated coordinates $\tilde{X} = \dfrac{\tilde{x}+\tilde{y}}{\sqrt{2}}, \tilde{Y} = \dfrac{\tilde{x}-\tilde{y}}{\sqrt{2}}$, they can be rewritten into
	\[
	G_1= \frac{\tilde{X}^2}{\frac{2(B^2-a^2)}{(a - 1)(B + 1)(B + a)}} +  \frac{\tilde{Y}^2}{\frac{-2(B^2-a^2)}{(a + 1)(B + 1)(B - a)}} -1
	\]
	and 
	\[
	G_2=\frac{\tilde{X}^2}{\frac{2(B^2-a^2)}{(a - 1)(B - 1)(B - a)}} +  \frac{\tilde{Y}^2}{\frac{-2(B^2-a^2)}{(a + 1)(B - 1)(B + a)}} -1.
	\]
	{Notice that $G_1=0$ contains no real points, since the coefficients of $\tilde{X}^{2}, \tilde{Y}^{2}$ are both negative.} Hence, only $G_2=0$ determines centered conic sections in $V$. We now compute the positions of their foci {by taking the affine change of the norm on $V$ into account}. Suppose that the foci of $G_2=0$ are located at $(\tilde{X},\tilde{Y})= (\pm c,0)$, then the norm $\| \cdot \|_c$ in $V$ which depends on the positions of foci is necessarily defined as
	\[
	\|(\tilde{X},\tilde{Y})\|^2_c = \frac{\tilde{X}^2}{1+c^2} + \tilde{Y}^2.
	\]
	This means we have the following equation in terms of $c$:
	\[
	\frac{c^2}{1+c^2} = \frac{\frac{2(B^2-a^2)}{(a - 1)(B - 1)(B - a)}}{1+c^2} - \frac{-2(B^2-a^2)}{(a + 1)(B - 1)(B + a)}.
	\]
	By solving this with respect to $c$, we obtain
	\[
	c = \pm \frac{2 \sqrt{a}}{1-a}
	\]
	which depends only on $a$. Therefore, the equation $G_2=0$ determines a family of confocal central conic sections in $V$. 
	From this fact and the projective correspondence of the spherical Kepler problem and the planer Kepler problem, we conclude that the equation \eqref{eq: confocal_central_conics_S} determines confocal centered spherical conic sections on $\mathcal{S}_{SH}$. 
	
	{Should we start from a family of confocal centered hyperbolae in $\hat{V}$ instead of ellipses, then we get the same type of results in a similar way. we thus conclude} that a family of confocal focused hyperbolic conic sections are transformed into a family of confocal centered spherical conic sections. 
	
	Analogously, one can show that a family of confocal focused hyperbolic conic sections are transformed into a family of confocal centered hyperbolic conic sections. 
\end{proof}

%We call the restriction of the spherical Hooke problem on the preimage in $\mathcal{S}_{SH}$ of the open unit disc  in $V$ by the stereographic projection the restricted spherical Hooke problem.

Combining these results, we obtain the following proposition.

\begin{prop} 
	The hyperbolic Kepler billiards with {a combination of branches of confocal conic sections} focused at the ``South pole'' $(0,0,-1)$ on the hyperboloid model as reflection wall are conformally transformed into the {hemispherical}/hyperbolic Hooke billiards with {the corresponding combination of confocal conic sections reflection wall} centered at the ``South pole'' on the {hemisphere}/hyperboloid. 
	Therefore their integrabilities are equivalent by \cite{Takeuchi-Zhao}.
\end{prop}

%%%%%%%%%%%%%%%%%%%%%%%%%%%%%%%%%%%%%%%%%%%%%%%%%%%%%%%%%%%%
{\bf Acknowledgement}	
A.T. is supported by Masason Foundation. L.Z. is supported by DFG ZH 605/1-1. 
%%%%%%%%%%%%%%%%%%%%%%%%%%%%%%%%%%%%%%%%%%%%%%%%%%%%%%%%%%%%

\vspace{1.5cm}

\hspace{-1cm}
\begin{tabular}{@{}l@{}}%
	Airi Takeuchi\\
	\textsc{Karlsruhe Institute of Technology, Karlsruhe, Germany.}\\
	\textit{E-mail address}: \texttt{airi.takeuchi@partner.kit.edu}
\end{tabular}
\vspace{10pt}

\hspace{-1cm}
\begin{tabular}{@{}l@{}}%
	Lei Zhao\\
	\textsc{University of Augsburg, Augsburg, Germany.}\\
	\textit{E-mail address}: \texttt{lei.zhao@math.uni-augsburg.de}
\end{tabular}
\end{document}